\documentclass[11pt,reqno]{amsart}
\usepackage[margin=1.2in]{geometry}

\usepackage{graphicx}
\usepackage{enumerate}
\usepackage{url}
\usepackage[shortlabels]{enumitem}
\usepackage{hyperref}
\usepackage{comment}
\usepackage{multirow}
\usepackage{amssymb}
\usepackage{amsmath}
\usepackage{amsxtra}

\newtheorem{prop}{Proposition}[section]
\newtheorem{thm}[prop]{Theorem}
\newtheorem{coro}[prop]{Corollary}
\newtheorem{lemma}[prop]{Lemma}
\newtheorem{conj}{Conjecture}

\newtheorem{remark}[prop]{Remark}
\numberwithin{table}{section}

\DeclareMathOperator{\Aut}{Aut}

\DeclareMathOperator{\Gal}{Gal}

\DeclareMathOperator{\cond}{cond}
\DeclareMathOperator{\Spec}{Spec}

\newcommand{\Om}{{\mathcal{O}}}
\newcommand{\Disc}{{\mathcal{D}}}

\def\ZZ{\mathbb Z}
\def\Z{{\mathbb Z}}             

\def\Q{{\mathbb Q}}             

\def\FF{\mathbb F}
\def\F{\mathbb F}
\def\C{{\mathbb C}}             

\def\<#1>{{\left\langle{#1}\right\rangle}}

\def\id#1{{\mathfrak{#1}}}      

\DeclareMathOperator{\norm}{{\mathcal N}}

\let\kro\dkro
\def\normid#1{{\norm{\id{#1}}}}

\def\p{{\mathfrak p}}
\def\q{{\mathfrak q}}

\DeclareMathOperator{\GL}{GL}
\DeclareMathOperator{\SL}{SL}

\DeclareMathOperator{\PGL}{PGL}

\def\vp{v_{\p}}
\def\vq{v_{\q}}

\newcommand{\lmfdbec}[3]{\href{http://www.lmfdb.org/EllipticCurve/Q/#1/#2/#3}{{\text{\rm#1.#2#3}}}}
\newcommand{\lmfdbeciso}[2]{\href{http://www.lmfdb.org/EllipticCurve/Q/#1/#2}{\text{\rm#1.#2}}}

\newcommand{\lmfdbecnf}[4]{\href{http://www.lmfdb.org/EllipticCurve/#1/#2/#3/#4}{{\text{\rm#1-#2-#3#4}}}}
\newcommand{\lmfdbecnfiso}[3]{\href{http://www.lmfdb.org/EllipticCurve/#1/#2/#3}{#1-#2-#3}}

\title[Elliptic curves over imaginary quadratic fields]
      {On elliptic curves of prime power conductor over imaginary
        quadratic fields with class number one}

\author{John Cremona}
\address{Warwick Mathematics Institute, University of Warwick,
  Coventry, CV4 7AL, UK}
\email{j.e.cremona@warwick.ac.uk}
\thanks{JEC is supported by EPSRC Programme Grant EP/K034383/1
  \textit{LMF: L-Functions and Modular Forms}, and the Horizon 2020
  European Research Infrastructures project \textit{OpenDreamKit}
  (\#676541).}

\author{John Cremona and Ariel Pacetti}
\address{FaMAF-CIEM, Universidad Nacional de C\'ordoba. C.P:5000, C\'ordoba, Argentina.}
\email{apacetti@famaf.unc.edu.ar}
\thanks{AP was supported by a Leverhulme Trust Visiting Professorship
 and by grant PIP 2014-2016 11220130100073.}


\begin{document}

\begin{abstract}
  The main result of this paper is to extend from $\Q$ to each of the
  nine imaginary quadratic fields of class number one a result of
  Serre (1987) and Mestre-Oesterl\'e (1989), namely that if $E$ is an
  elliptic curve of prime conductor then either $E$ or a $2$-, $3$- or
  $5$-isogenous curve has prime discriminant.  For four of the nine
  fields, the theorem holds with no change, while for the remaining
  five fields the discriminant of a curve with
  prime conductor is (up to isogeny) either prime or the square of a prime.  The proof is
  conditional in two ways: first that the curves are modular, so are
  associated to suitable Bianchi newforms; and second that a certain
  level-lowering conjecture holds for Bianchi newforms.  We also
  classify all elliptic curves of prime power conductor and
  non-trivial torsion over each of the nine fields: in the case of
  $2$-torsion, we find that such curves either have CM or with a small
  finite number of exceptions arise from a family analogous to the
  Setzer-Neumann family over $\Q$.
\end{abstract}

\maketitle


\section*{Introduction}

The theory of elliptic curves plays a crucial role in modern number
theory. An important advance came with the systematic construction of
tables, as done by the first author for elliptic curves defined over
$\Q$ (\cite{CremonaAlgo}). The original purpose of the present article
was to prove analogues over imaginary quadratic fields of class number
one to well-known bounds relating the conductor and discriminant of an
elliptic curve over $\Q$ with prime power conductor. A useful
application of such a result is an effective algorithm to construct
tables of elliptic curves of prime power conductor, via solving Thue
equations. While succeeding in carrying out our original aim, we also
found several examples of phenomena for curves over imaginary
quadratic fields that do not occur over $\Q$, and were able to
classify all elliptic curves with prime power conductor and nontrivial
torsion over the fields in question, extending the results of Miyawaki
\cite{Miyawaki} for elliptic curves over $\Q$ and Shumbusho
\cite{Shumbusho} over imaginary quadratic fields.

Over $\Q$, there are only finitely many elliptic curves with complex
multiplication (CM curves) of prime power conductor $p^r$. The reason
is, first, that there are finitely many imaginary quadratic orders
$\Om$ of class number $1$, and second, that if $p$ divides the
discriminant of $\Om$, then all curves~$E$ with endomorphism ring
isomorphic to~$\Om$ have additive reduction at $p$, and the field
where $E$ attains good reduction is not abelian. In particular, any
other rational curve isomorphic to~$E$ (that is, a twist of~$E$) will
also have bad reduction at $p$, so the prime power discriminant
condition gives at most one curve per order.

Over an imaginary quadratic field~$K$, this finiteness statement no
longer holds true. If $E/K$ has CM by $\Om_K$, the ring of integers of
$K$, and $p$ ramifies in $K$, the curve attains good reduction over an
abelian extension of the completion $K_p$. In particular, some twist
of $E$ (quadratic, quartic or sextic, depending on the number of roots
of unity in $K$) attains good reduction at $p$. However the required
local extension does not come from a global one, so we do not get a
curve over~$K$ of conductor $1$ (which do not exist), but an infinite
family, all twists of each other, of CM curves of prime square
conductor; moreover in each case the density of primes of $K$ whose
squares arise in this way is in each case $1/\#\Om_K^*$.  In
Section~\ref{section CM}, we explain this phenomenon, and describe
precisely all elliptic curves of prime power conductor and complex
multiplication.

Another notable new phenomenon that occurs over imaginary quadratic
fields, is the existence of elliptic curves of prime conductor whose
residual Galois representations modulo $p$ are irreducible but not
absolutely irreducible. When $p$ is odd, this phenomenon does not
occur for elliptic curves over~$\Q$: the reason is the existence of
complex conjugation in the absolute Galois group, whose image under
the mod-$p$ representation is similar to
$\left(\begin{smallmatrix}1 & 0\\ 0 & -1 \end{smallmatrix}
\right)$. Over an imaginary quadratic field, or any number field which
is totally complex (i.e. has no real embedding), the absolute Galois
group has no such complex conjugation elements and this argument does
not apply.

At the prime $2$ the situation is more exciting! Over~$\Q$, if $E$ is
a semistable elliptic curve whose discriminant is a square, then its
residual image is reducible. There are two different ways to prove
such an assertion: one comes from the study of finite flat group
schemes over $\Spec(\Z)$. The hypotheses ($E$ being semistable and the
discriminant being a square) imply by a theorem of Mazur (see
Theorem~\ref{thm:Mazur} below) that $E[2]$ is a finite flat group
scheme of type $(2,2)$. As explained in \cite{Mestre}, there are only
$4$ such group schemes, and only two of them satisfy that the
determinant of the group scheme is isomorphic to the group scheme
$\mu_2$; they both have an invariant subspace. A different approach
comes from the study of the mod-$2$ Galois representation itself: if
it is irreducible, then by Ribet's level-lowering result, there would
exist a modular form of level $1$ and weight $2$, while if the image
were irreducible but not absolutely irreducible (i.e. it is the cyclic
subgroup of order $3$ in $\GL_2(\F_2)$), then there would exist a
cubic extension of $\Q$ unramified outside $2$; but such an extension
does not exist.

Over an imaginary quadratic field~$K$ of class number one, in four
cases such an extension does not exist either, so the same result
holds (for modular elliptic curves, assuming an analogous
level-lowering conjecture).  However for $K=\Q(\sqrt{-d})$ for
$d\in\{11,19,43,67,163\}$ where $2$ is inert in $K$, there is no
reason for such Galois representation not to exist.  In particular,
they supply new examples of irreducible finite flat group schemes of
type $(2,2)$ over $\Spec(\Om_K)$ (the irreducibility comes from the
fact that irreducible $2$-groups are only the additive and the
multiplicative ones as proved in \cite[Corollary page 21]{Oort-Tate},
which give rational $2$-torsion).  One simple example of this
phenomenon comes from the elliptic curves over~$\Q$ of conductor~$11$,
which have minimal discriminant either $-11$ or $-11^5$; after
base-change to $K=\Q(\sqrt{-11})$ the conductor is still prime but the
minimal discriminants are all now square.  For an example which is not
a base-change, the curve \lmfdbecnf{2.0.11.1}{47.1}{a}{1} (using its
LMFDB label, see \cite{lmfdb}):
\[
  E:\quad y^2+y=x^3+\alpha x^2-x
  \]
where $\alpha=(1+\sqrt{-11})/2$, has prime conductor $\p=(\pi)$ of
norm~$47$, with $\pi=7-2\alpha$, and discriminant $\pi^2$, and is
alone in its isogeny class.  This phenomenon is studied in detail in
Section~\ref{section:evenexponents} below. We suspect that there are
infinitely many examples like this over each of these five fields.

Another interesting problem that arose is that of classifying elliptic
curves (up to $2$-isogenies and twists) with a rational $2$-torsion
point. The prime $2$ is always hard to handle, and we develop
tools (including a result characterizing such curves, see
Theorem~\ref{thm:classification2torsion}) that allow us to give a
complete description of them. The ideas developed
here could be adapted to more general situations, but the fact that
$K$ has finitely many units appears to play an important role.

The main results of the article are Theorem~\ref{thm:mainthm} and
Corollary~\ref{coro:Szpiro}, which establish Szpiro's conjecture for
``modular'' prime power conductor elliptic curves over $K$, assuming a
form of level-lowering result
(Conjecture~\ref{conj:loweringlevel}). The general strategy of the
proof is as follows: let $E/K$ be a modular elliptic curve (see
Conjecture~\ref{conj:modularity}) of prime power conductor. If $E$ has
potentially good reduction, there is a well-known bound for its
minimal discriminant~$\Disc(E)$, so we can focus on the semistable case. If
$\ell$ is a rational prime dividing the valuation of $\Disc(E)$, by a
theorem of Mazur, $E[\ell]$ is a finite flat group scheme. If the
residual Galois representation is absolutely irreducible, then the
level-lowering conjecture implies the existence of a Bianchi modular
form modulo~$\ell$ of level $1$ and weight $2$ whose Galois
representation matches that of $E$. Since for $\ell$ odd, there are no such forms over
the fields in question, we get a contradiction. We want to
emphasise that the level-lowering result might be hard to prove for
small primes (which divide the order of elliptic points). For this
reason, we include an unconditional proof (not relying on modularity) for $\ell=2$ and $3$ in Theorem~\ref{thm:mainthm} for the nine fields considered.

Assume otherwise that the residual image is absolutely reducible, i.e. it is
either reducible over~$\F_{\ell}$, or irreducible but reducible
over~$\F_{\ell^2}$. For the first possibility, we prove that either
such curves do not exist, or otherwise there is an isogenous curve
whose discriminant valuation is prime to $\ell$ (the case $\ell=2$
being hardest). When $\ell$ is odd, the second case is eliminated by a
detailed study of the residual Galois representation.  Lastly, for
$\ell=2$ we prove that the discriminant valuation is at most $2$, where
(for certain of the fields~$K$ only) the exceptional curves with
square discriminant arise as described above.

The article is organized as follows. The first section contains a brief
description of Bianchi modular forms and modularity of elliptic
curves. It also contains the conjectural level-lowering statement in the
spirit of Ribet's result, which we expect to hold in the setting of
Bianchi modular forms. An important difference in our statement comes
from the fact that we do not expect forms of minimal level to always
lift to characteristic zero: this is why the statement of
Conjecture~\ref{conj:loweringlevel} is in terms of group cohomology
over~$\F_{\ell}$.

The second section contains the main theorem, and its proof when the
image is absolutely irreducible. The third section studies elliptic
curves over $K$ with complex multiplication. The main results includes
a complete classification of all CM curves of odd prime
power discriminant.  The case of small image at an odd prime is
treated in the fourth section, where elliptic curves over $K$ of odd
prime power conductor with a rational point of odd order~$\ell$ are
considered. In Theorem \ref{thm:3-torsionspecialcase} we prove that if
$\ell=3$ and $K=\Q(\sqrt{-3})$, then any such curve either has
discriminant valuation not divisible by $3$, or its $3$-isogenous curve
satisfies this property. For all other choices of $\ell$ and $K$, we
give a finite list of possible curves: see
Table~\ref{table:oddtorsioncurves}.

The fifth section studies curves of odd conductor over $K$ with a
rational $2$-torsion point. An important result is
Theorem~\ref{thm:classification2torsion}, where a description of all
such curves (up to twist) is given. This result is very general, and
can be applied to different situations. Using this characterization,
we describe all curves of prime power conductor $\id{p}^r$ with a
rational $2$-torsion point, according to the following possible cases:
either a twist of the curve has good reduction at $\id{p}$, or all
twists have additive reduction at $\id{p}$, or the curve has a twist
of multiplicative reduction. We prove that in the first case all
curves have CM and come in families, described explicitly in
Theorem~\ref{thm:goodtwist}, using the results from
Section~\ref{section CM}.  For the additive ones, we prove in
Theorem~\ref{thm:additivetwist} that there are only some sporadic
cases (also all CM), and for multiplicative ones, we prove that there
are potentially (and probably) infinitely many, almost all belonging
to a family analogous to the so-called \emph{Setzer-Neumann} family
(as in \cite{Setzer}) for curves defined over $\Q$. As in the
classical case, all such curves have rank $0$.

The important consequence of this detailed case by case study for our
main result is that if $E$ is an elliptic curve defined over $K$ of
odd prime conductor with a rational $2$-torsion point, then
either $E$ or a $2$-isogenous curve has odd discriminant valuation.

Finally, in the last section, we study the case of curves whose Galois
image modulo~$2$ is cyclic of order $3$, giving curves of prime
conductor but prime square discriminant, as explained above.

\subsection*{Notation and terminology.} $K$ will denote an imaginary
quadratic field $\Q(\sqrt{-d})$ of class number~$1$, with ring of
integers $\Om_K$.  As is well known, there are~$9$ such fields, with
$d\in\{1,2,3,7,11,19,43,67,163\}$.  Many of the results of
Section~\ref{section 2 torsion} also apply to the case $K=\Q$.

We say that an ideal or element of $\Om_K$ is \emph{odd} if it is
coprime to~$2$.  Let $e_2$ be the ramification degree of $2$ in
$K/\Q$, so that $e_2=1$ except for $K=\Q(\sqrt{-1})$ and
$K=\Q(\sqrt{-2})$.  Note also that $2$ splits only in
$K=\Q(\sqrt{-7})$.  The primes dividing $2$ play a crucial role in
Section~\ref{section 2 torsion}, where we will denote by $\q$ a prime
dividing $2$, and by $\p$ any prime ideal (which might divide $2$ or
not).  The valuation at~$\p$ is denoted~$\vp()$.  We denote
by~$\varepsilon$ a generator of the finite unit group~$\Om_K^*$ so
that $\varepsilon=-1$ except for $K=\Q(\sqrt{-1})$ and
$K=\Q(\sqrt{-3})$ when $\varepsilon=\sqrt{-1}$, or $\varepsilon$ is
a $6$th root of unity, respectively.

All curves explicitly mentioned will be labeled with their LMFDB label, see
  \cite{lmfdb}.

\begin{remark}
  It is well-known that given an elliptic curve $E/K$, if the class
  number of $K$ equals $1$ then $E$ has a global minimal model. Such
  model might not be unique, hence there is in general no notion of a
  \emph{minimal discriminant}. However, over an imaginary quadratic
  field, all units are annihilated by $12$, and hence the value of the
  minimal discriminant~$\Disc(E)$ is well-defined.
\end{remark}

\medskip

\textbf{Acknowledgments:} we would like to thank Nicolas Vescovo for
participating in some discussions of the present article, to Luis
Dieulefait, for explaining to us some technicalities used in the proof of
Theorem \ref{thm:Kraus}, to Angelos Koutsianas for his computations
used in the proof of Theorem~\ref{thm:mainthm}, and to Samir Siksek
and Haluk {\c S}eng\"un for many useful conversations. The second
author would like to thank the University of Warwick for its
hospitality during his visit as a Leverhulme Visiting Professor. Last
but not least, we would like to thank the anonymous referee for many
useful comments and corrections.

\section{Bianchi modular forms and modularity of elliptic curves}
For background on Bianchi modular forms and modularity of elliptic
curves defined over imaginary quadratic fields, we refer to the survey
article of {\c S}eng\"un \cite{Sengun} and the work of the first
author (\cite{CremonaTessellations}, \cite{CremonaTwist},
\cite{Cremona-Whitley}).  For our purposes we may restrict our
attention to the space $S_2(\id{n})$ of Bianchi modular forms which
are cuspidal and of weight~$2$ for the congruence subgroup
$\Gamma_0(\id{n})\leq\GL_2(\Om_K)$, where the level
$\id{n}\subseteq\Om_K$ is an integral ideal of~$K$.  This space is a
finite-dimensional vector space equipped with a Hecke action. Its
newforms subspace is spanned by eigenforms which are simultaneous
eigenvectors for the algebra of Hecke operators.  Bianchi modular
forms can also be seen within the context of cohomological automorphic
forms, since we have the isomorphism
$S_2(\id{n})\cong H^1(Y_0(\id{n}),\C)$, where $Y_0(\id{n})$ is the
quotient of hyperbolic $3$-space $\mathcal{H}_3$
by~$\Gamma_0(\id{n})$.  A more concrete description of these Bianchi
modular forms is as real analytic functions $\mathcal{H}_3\to\C^3$
satisfying certain conditions.  In the work of the second author and
his students (\cite{CremonaTessellations}, \cite{Cremona-Whitley})
explicit methods were developed to compute the spaces $S_2(\id{n})$
over~$K$ for each of the nine imaginary quadratic fields~$K$ of class
number~$1$. Recall that the orders of the isotropy groups of points of
$Y_0(1)$ are only divisible by the primes $2$ and $3$, hence for
$\ell \ge 5$ the cohomology $H^1(Y_0(1),\F_\ell)$ matches the group
cohomology
$H^1(\PGL_2(\Om_K),\F_\ell) = H^1(\PGL_2(\Om_K),\Z)\otimes \F_\ell$
(see for example \cite[Section 1.4]{Ash-Stevens}). Then the previous
computations, together with results of \cite{Haluk-expmath}, establish
the following result:
\begin{thm}
For each of the nine imaginary quadratic fields of class number~$1$,
the space $S_2(1)$ of weight~$2$ cuspidal Bianchi modular forms of
level~$1$ is trivial.  Moreover for all primes $\ell\ge 5$ the space
$H^1(Y_0({1}),\F_{\ell})$ is also trivial.
\label{thm:noleveloneforms}
\end{thm}

It is known that these Bianchi modular forms have associated
$\ell$-adic Galois representations $\rho_{F,\ell}$.  These were first
constructed by Taylor \textit{et al.}~in \cite{TaylorI},
\cite{TaylorII} with subsequent results by Berger and Harcos
in~\cite{Berger}.  Below we only need to refer to the residual mod-$\ell$
representations $\overline{\rho_{F,\ell}}$.

For an elliptic curve~$E$ defined over an imaginary quadratic
field~$K$, we say that $E$ is modular if $L(E,s)=L(F,s)$ for some
$F\in S_2(\id{n})$ over~$K$, where $\id{n}$ is the conductor of~$E$.
The following conjecture, a version of which was first made by
Mennicke, is part of Conjecture~9.1 of \cite{Sengun}, following
\cite{CremonaTwist}:
\begin{conj}
Let $E$ be an elliptic curve defined over an imaginary quadratic
field~$K$ of class number~$1$ that does not have complex
multiplication by an order in~$K$.  Then $E$ is modular.
\label{conj:modularity}
\end{conj}
Some cases of the conjecture have been proved (see \cite{DGP}), and
under some mild hypothesis it can be shown that any such curve is
potentially modular (see \cite[Theorem 1.1]{Calegari-Geraghty}).
Curves with complex multiplication correspond to Hecke characters and
belong to non-cuspidal modular forms. These will be considered in
Section~\ref{section CM}.

Following the classical result of Ribet
(\cite{Ribetlowering}), we make the following conjecture.

\begin{conj}
  Let $F$ be a weight $2$ Bianchi newform of level
  $\Gamma_0(\id{n}\id{p})$, with $\id{p}$ a prime number. Suppose that
  $\ell\ge 5$ is a prime for which the representation $\rho_{F,\ell}$
  satisfies:
  \begin{enumerate}
  \item The residual representation is absolutely irreducible.
  \item The residual representation is finite at $\id{p}$.
  \end{enumerate}
  Then there is an automorphic form $G \in H^1(Y_0(\id{n}),\F_{\ell})$ whose Galois
  representation is isomorphic to $\overline{\rho_{F,\ell}}$.
\label{conj:loweringlevel}
\end{conj}
Some results in the direction of the conjecture are proven in
\cite{Calegari}. Note that the conjecture refers to an element in the
mod $\ell$ cohomology. In contrast to modular forms for $\SL_2(\ZZ)$,
for Bianchi modular forms the homology (and the cohomology) of
$Y_0(\id{n})$ contains a big torsion part. By results of Scholze
(\cite{Scholze}) the torsion forms do have Galois representations
attached, but these are residual representations that in general do
not admit a lift of level $\id{n}$ to characteristic zero, so the
form in Conjecture~\ref{conj:loweringlevel} is not expected to
be global.

\section{Modular elliptic curves of prime power conductor}

Recall the following result due to Serre and Mestre-Oesterl\'e (see \cite{Mestre}).
\begin{thm} Let $E/\Q$ be a modular elliptic curve of prime conductor
  $p$. If $p >37$ then $\Disc(E) = \pm p$ up to $2$-isogenies (i.e.,
  either $\Disc(E) = \pm p$, or there exists a curve
  $2$-isogenous to~$E$ over~$\Q$ with discriminant~$\pm p$).
\label{thm:MO}
\end{thm}
A key ingredient in the proof is the following result of Mazur.

\begin{thm}[(Mazur)] Let $K$ be a number field, let $E/K$ be a
  semistable elliptic curve with discriminant~$\Disc(E)$, and let
  $\ell$ be a prime number. Then $E[\ell]$ is a finite flat group
  scheme if and only if $\ell \mid v_{\id{q}}(\Disc(E))$ for all
  primes $\id{q}$ in $\Om_K$, i.e. if the ideal $(\Disc(E))$ is an
  $\ell$-th power.
\label{thm:Mazur}
\end{thm}
\begin{proof}
  See \cite[Proposition 9.1]{Mazur2}.
\end{proof}

If $E/K$ is an elliptic curve of prime conductor~$\id{p}$, and $\ell$
is a prime dividing the valuation $v_{\id{p}}(\Disc(E))$ of $\Disc(E)$,
then $E[\ell]$ is a finite flat group scheme over $\Spec(\Om_K)$ by
Mazur's theorem. To apply a level-lowering result such as
Conjecture~\ref{conj:loweringlevel}, one also needs a \emph{big image}
hypothesis, i.e. that the residual representation at $\ell$ is
absolutely irreducible. While working over the rationals, either the
residual image is reducible, in which case the curve has a 
$\Q$-point (which does not happen for $\ell>7$) or it is absolutely
irreducible. The reducible case for small primes can be discarded by a
result of Fontaine (\cite[Theorem B]{Fontaine}).

One of the main results of this article is a generalization of
Theorem~\ref{thm:MO} to $K$.
\begin{thm}
  Let $K$ be an imaginary quadratic field of class number $1$. Let
  $E/K$ be a modular elliptic curve of prime conductor. Assume
  Conjecture \ref{conj:loweringlevel} holds. Then there exists a curve
  isogenous to~$E$ over~$K$ with prime or prime square discriminant.
\label{thm:mainthm}
\end{thm}
\begin{proof}
Let $\id{p}$ be the conductor of $E$.  For the case $\id{p}\mid 2$,
A.~Koutsianas used the methods of \cite{Angelos} to compute for us all
elliptic curves with conductor a power of $\id{p}$, and found that
over each of the nine fields there are none with conductor $\id{p}$.
This is consistent with the tables of automorphic forms of
\cite{CremonaTessellations} and \cite{Cremona-Whitley}.

From now on we may assume that $\id{p} \nmid 2$. Suppose that
$\Disc(E)$ is not prime, and let $\ell$ be a prime number dividing
$v_{\id{p}}(\Disc(E))$.  Recall that by semistability and Mazur's
result, the Galois module $E[\ell]$ is unramified away from~$\ell$,
and also that its determinant is the $\ell$th cyclotomic character. We
claim that the hypotheses of the theorem imply that for every prime
number $\ell$ dividing $v_{\id{p}}(\Disc(E))$, the module \textit{$E[\ell]$ is
  not absolutely irreducible}.

Assume the claim.  For odd $\ell$, Theorem~\ref{thm:Kraus} below implies
that either~$E$, or an $\ell$-isogenous curve, has a rational torsion
point of order~$\ell$. In Section \ref{smallimage} all curves of odd
prime conductor with an $\ell$-torsion point are computed (see
Table~\ref{table:oddtorsioncurves}) with the exception of curves with
a rational $3$-torsion point over $\Q(\sqrt{-3})$. It is easy to
verify from Table~\ref{table:oddtorsioncurves} that all curves of
prime conductor listed there have (up to an $\ell$-isogeny)
discriminant with valuation prime to $\ell$. Finally,
Theorem~\ref{thm:3-torsionspecialcase} proves that any curve with a
rational $3$-torsion point over $\Q(\sqrt{-3})$ either has
discriminant valuation prime to $3$, or a $3$-isogenous curve does.
Hence, up to isogeny, the claim (for all odd $\ell$) implies that
$v_{\id{p}}(\Disc(E))$ has no odd prime factors.

Suppose that $v_{\id{p}}(\Disc(E))$ is even.  The hypothesis implies
that $\Disc(E)=\varepsilon\pi^2$ for some~$\varepsilon\in\Om_K^*$ and
$\pi\in\Om_K$.  If $E[2]$ is absolutely irreducible, then $\Disc(E)$
is not a square, and the extension $K(\sqrt{\varepsilon})/K$ has
degree~$2$ and is unramified away from~$2$.  However, using explicit
class field theory (as implemented in \cite{PARI2}), we may check that
none of the nine fields~$K$ has an extension~$L/K$ with Galois group
$\GL_2(\F_2)\cong S_3$ which is unramified away from~$2$ and has
quadratic subfield $K(\sqrt{\varepsilon})$.  Hence $E[2]$ is not
absolutely irreducible, establishing the claim for $\ell=2$.  Now
either $E[2]$ is reducible, in which case
Corollary~\ref{coro:2torsion-odddiscriminant} shows that up to
$2$-isogeny $v_{\id{p}}(\Disc(E))=1$; or $E[2]$ is irreducible (but
not absolutely irreducible) in which case Theorem~\ref{thm:absred2}
implies that $v_{\id{p}}(\Disc(E))=2$.

It remains to prove the claim for odd primes~$\ell$ which divide
$v_{\id{p}}(\Disc(E))$.  We consider first the case $\ell=3$, where as
with $\ell=2$ an elementary argument establishes the claim
unconditionally.  If $K\ne\Q(\sqrt{-3})$, then every unit in~$K$ is a
cube, so the hypothesis that $3\mid v_{\id{p}}(\Disc(E))$ implies that
$\Disc(E)$ is a cube in~$K$.  In case $K=\Q(\sqrt{-3})$ we can only
say that $\Delta(E)$ is a unit times a cube, and we may have
$K(\sqrt[3]{\Disc(E)})=\Q(\zeta_9)$.  In general,
$K(\sqrt[3]{\Disc(E)})$ is the subfield of $K(E[3])$ cut out by the
Sylow $2$-subgroup, which has index~$3$ and order~$16$, in
$\GL_2(\F_3)$; this subgroup is the normaliser of the non-split Cartan
subgroup.  Hence the field $K(E[3])$ is obtained by a succession of
quadratic extensions of $K(\sqrt[3]{\Disc(E)})$ each unramified away
from~$3$.  Neither $\Q(\sqrt{-3})$ nor $\Q(\zeta_9)$ possesses any
such quadratic extension, so for $K=\Q(\sqrt{-3})$ the image of the
mod-$3$ representation is either trivial (when $\Delta$ is a cube) or
of order~$3$, in which case it is conjugate to
$\left(\begin{smallmatrix} 1 & *\\ 0 & 1\end{smallmatrix} \right)$.

For the other fields, the residual image is contained in the
normaliser of the non-split Cartan subgroup, of order~$16$.  If $E[3]$
were absolutely irreducible then the residual image could not be
contained in the non-split Cartan subgroup itself, since that is
absolutely reducible; hence there would be a (non-trivial) quadratic
character unramified outside~$3$ whose kernel has residual image equal
to the non-split Cartan subgroup.  However, for each of the eight
fields~$K$ in question, the only quadratic extension ramified only
above $3$ is $K(\zeta_3)$, which cuts out a \textit{different}
index~$2$ subgroup of the non-split Cartan normaliser, since the
non-split Cartan contains matrices of determinant~$-1$. (Here we are
using the fact that the determinant of the representation is the
cyclotomic character.)  This contradiction shows that $E[3]$ is not
absolutely irreducible.

Finally, we prove the claim when $v_{\id{p}}(\Disc(E))$ is divisible
by a prime~$\ell\ge5$.  By Theorem~\ref{thm:Mazur}, $E[\ell]$ is a
finite flat group scheme. If the residual representation of $E$ at the
prime $\ell$ is absolutely irreducible, the modularity assumption on
$E$ gives the hypothesis of Conjecture \ref{conj:loweringlevel}, so
there must exist a mod $\ell$ Bianchi modular form of level $1$ whose
Galois representation matches the residual representation of $E$
modulo $\ell$; this contradicts Theorem~\ref{thm:noleveloneforms}.
\end{proof}

As a corollary, we have the following version of Szpiro's conjecture.
\begin{coro}
  Let $K$ be an imaginary quadratic field of class number $1$. Assume
  that Conjecture~\ref{conj:loweringlevel} holds. Let $E/K$ be a
  modular elliptic curve of prime power conductor. Then
  $\norm(\Disc_{\min}(E)) \le \norm(\cond(E))^{6}$, except for the
  curve \lmfdbec{11}{a}{2} over $\Q(\sqrt{-11})$, whose conductor has
  valuation~$1$ while its discriminant has valuation~$10$.
\label{coro:Szpiro}
\end{coro}
\begin{proof}
  If $E$ has potentially good reduction, then the result is
  well-known, see for example \cite[p.~365, Table~4.1]{Silverman2}.
  Otherwise, either $E$ is semistable or is a
  quadratic twist of a semistable curve. Recall that there are no
  curves of prime conductor dividing $2$, so we can restrict to odd
  conductor. It is enough to consider the semistable case, since if
  $E$ has prime conductor $\id{p}$ and satisfies
  $\norm(\Disc_{\min}(E)) \le \normid{p}^{6}$, the twisted curve $E'$
  satisfies
  $\norm(\Disc_{\min}(E'))=\norm(\Disc_{\min}(E)) \cdot \normid{p}^6$
  and $\cond(E')=\id{p}^2$.

  By Theorem~\ref{thm:mainthm} the result holds for some curve in the
  isogeny class.  Furthermore, when the isogeny class of a semistable
  curve has at most $2$ elements, Theorem~\ref{thm:absred2} and
  Corollary~\ref{coro:2torsion-odddiscriminant} prove that the
  discriminant valuation is at most $2$.

  When there is a $3$-isogeny, the isogenous curve has discriminant
  valuation $3$ or $6$ (depending on whether or not the conductor
  is ramified in $K$), and there is a unique semistable case with a
  $5$-isogeny (see Theorem~\ref{thm:5-torsion}), which is precisely
  the curve \lmfdbec{11}{a}{2}.
\end{proof}
We end this section with a result used in the proof of
Theorem~\ref{thm:mainthm} above.

\begin{thm}Let $K$ be an imaginary quadratic field of class number $1$
  and $E/K$ be a semistable elliptic curve. Let $\ell \ge 3$ be a
  prime number such that the residual representation of $E$ at $\ell$
  is not absolutely irreducible.  Then either $E$, or a curve
  $\ell$-isogenous to~$E$ over~$K$, has a point of order $\ell$
  defined over $K$.
\label{thm:Kraus}
\end{thm}

\begin{proof} The residual representation associated to $E$ takes
  values in $\GL(2,\F_\ell)$, and may be either reducible
  (over~$\F_{\ell}$), irreducible but absolutely reducible (i.e.,
  reducible over $\F_{\ell^2}$) or absolutely irreducible.  The
  hypothesis in the theorem excludes only the absolutely irreducible
  case.  Over $\Q$, or any number field with at least one real place, the second case
  cannot occur, due to the action of complex conjugation: any
  invariant line over $\F_{\ell^2}$ must be defined over~$\F_{\ell}$.
  In our situation, we need to consider both the reducible and the
  absolutely reducible cases separately.

  \noindent $\bullet$ {\bf Reducible case:} without
  loss of generality, the residual representation is of the form
  \begin{equation}
    \label{eq:reducibledecomposition}
      \overline{\rho_{E,\ell}} \simeq \left(\begin{smallmatrix} \theta_1 & *\\ 0
    & \theta_2\end{smallmatrix} \right),
  \end{equation}
  for $\theta_i$ characters of $\Gal(\overline{K}/K)$ such that
  $\theta_1 \theta_2 = {\chi_{\ell}}$ (the cyclotomic
  character). Since $E$ is semistable, the conductors of the characters
  $\theta_i$ are supported in $\ell$ (see \cite[Lemme 1]{Kraus}).  Also since $K$ has class
  number~$1$, the only unramified character is the trivial character.
  Hence we must show that at least one of the characters is
  unramified, since $\theta_1$ trivial implies that $E$ has a point of
  order~$\ell$ while if $\theta_2$ is trivial then the isogenous curve
  has such a point.  If $\ell$ is unramified in $K$, then the result
  follows from the proof of \cite[Corollaire 1]{Kraus} (pages
  249-250), as we now explain.

  If $\ell$ is inert in~$K$ then only one of the $\theta_i$ can be
  ramified at $\ell$ (see \cite[Lemme 1]{Kraus}), hence the statement.
  Then we can restrict to the case when both characters are ramified at a prime
  dividing $\ell$.

  If $\ell = \id{l}_1 \id{l_2}$ splits, we can assume that $\theta_i$
  has conductor $\id{l}_i$ since by \cite[Lemme 1]{Kraus} they cannot
  both be ramified at the same prime. Then on one hand the restriction
  of $\theta_i$ to the inertia group at $\id{l}_i$ is the cyclotomic
  character, so $\theta_i(-1)=-1$, and on the other hand it is a character
  of $(\Om_K/\id{l}_i)^\times$ which is trivial in $\Om_K^\times$ (as
  it factors through the Artin map), so $\theta_i(-1)=1$. This is
  impossible since $\ell\not=2$.

  Lastly, suppose that $\ell$ ramifies in $K$, so $\ell \equiv 3 \pmod
  4$ and $K=\Q(\sqrt{-\ell})$. In particular, the restriction of the
  cyclotomic character to~$\Gal(\overline{K}/K)$ has (odd) order
  $(\ell-1)/2$.  Let $\id{l}$ denote the prime of $K$ dividing~$\ell$
  and $I_{\id{l}}$ the associated inertia subgroup. Then the
  $\theta_i$ are both characters of level $1$ at $\id{l}$ (since
  characters of level $2$ have irreducible image), $\theta_1 \theta_2
  = \chi_{\ell}$ and $E$ has good supersingular reduction at
  $\ell$. In the notation of \cite{Serre2}, let $a_\ell$ denote the
  $\ell$-th coefficient in the series for multiplication by $\ell$ in
  the formal group of $\tilde{E}$, the reduced curve over
  $\Om_K/\id{l}\cong\F_{\ell}$. By \cite[Proposition 10, page
    272]{Serre2}, if $v_\ell(a_\ell) > 1$, then
  $\theta_i|_{I_{\id{l}}}$ for $i=1,2$ would both be squares of
  fundamental characters of level $2$, whose image is not in
  $\FF_{\ell}$; hence the valuation is $1$, and the same proposition
  implies that $\theta_i|_{I_{\id{l}}}$ equals the fundamental
  character of level $1$ for $i=1,2$, and $*$ is unramified at
  $\ell$. So $\theta_1/\theta_2$ is unramified, hence trivial, and so
  $\theta_1 = \theta_2$ and $\theta_1^2=\chi_{\ell}$.

  The order of $\theta_1$ equals $\ell -1$ and it factors through a
  cyclic degree $(\ell-1)$ extension of $K$ unramified outside $\ell$,
  containing $\zeta_{\ell}$ (the $\ell$-th roots of unity). This implies in
  particular the existence of a quadratic extension of $K$
  unramified outside $\ell$. It can be easily verified (using
  \cite{PARI2} for example) that there are no such quadratic
  extensions, hence this case cannot occur.

  \medskip
  
  \noindent $\bullet$ {\bf Irreducible but absolutely reducible
    case}. In this case, there is a character of $K$ of order
  ${\ell^2-1}$ or $\frac{\ell^2-1}{2}$ unramified outside $\ell$ (by
  the same argument as before). This implies the existence of a quadratic
  extension of $K$ unramified outside $\ell$ which, as we saw in the previous case, does not exist.
\end{proof}

\section{Curves with complex multiplication}
\label{section CM}

Let $K$ denote one of the nine imaginary quadratic fields with class
number one.  Let $E/K$ be an elliptic curve with complex
multiplication by an order~$\Om$ in~$K$. Since $K$ has class number
one, the $j$-invariant $j(E)$ is rational, in fact in $\Z$, so $E$ is
a twist of the base extension of an elliptic curve defined over $\Q$.

For each field~$K$, we fix one such curve $E$, choosing it to have bad
reduction only at the unique ramified prime $\id{p}$ of~$K$, and to
have endomorphism ring isomorphic to the maximal order~$\Om_K$.  Every
other elliptic curve with CM by an order in $K$ is then isogenous to a
twist of this base curve~$E$.  For $d=1,2,3,7,11,19,43,67,163$
respectively we take the base curve to be the base-change to $K$ of
the elliptic curve over $\Q$ with LMFDB label \lmfdbec{64}{a}{4},
\lmfdbec{256}{d}{1}, \lmfdbec{27}{a}{4}, \lmfdbec{49}{a}{4},
\lmfdbec{121}{b}{2}, \lmfdbec{361}{a}{2}, \lmfdbec{1849}{b}{2},
\lmfdbec{4489}{b}{2}, \lmfdbec{26569}{a}{2}, respectively.  Our goal
in this section is to determine which twists of these base curves $E$
have odd prime power conductor, recalling that for $d=1$ and $d=3$
respectively, we must consider quartic and sextic twists, not only
quadratic twists.

\begin{remark}
For $d>3$ the base curve listed is uniquely determined (up to
isomorphism over~$K$) by the condition that it has CM by $\Om_K$ and
bad reduction only at the ramified prime $\id{p}=(\sqrt{-d})$.
However for $d=1,2,3$ there are several choices.  In the results of
this section, we consider elliptic curves with prime power conductor
as explicit twists of the base curve, so it is important to fix this
choice.
\end{remark}

The automorphic form attached to $E/K$ consists of the sum of two
conjugate Hecke Grossencharacters $\{\chi,\bar{\chi}\}$ of infinity
types $(1,0)$ and $(0,1)$, and conductor $\id{n}$ which is a power
of~$\id{p}$. In particular, $\cond(E) = \id{n}^2$.  Note that the
Grossencharacters take values in $K^\times$.  The character $\chi$ is
unramified outside $\id{p}$, so has conductor a power of~$\id{p}$, and
the local character $\chi_{\id{p}}$ restricted to
$\Om_{\id{p}}^\times$ is a finite character taking values in the roots
of unity of $K$. In particular, it is quadratic except for $d=-1,
-3$.

\subsection{$K=\Q(\sqrt{-d})$ with $d \neq 1, 3$} The character $\chi$
is quadratic, and unramified outside the prime $\id{p}=(\sqrt{-d})$,
the unique ramified prime of $K$. If we twist $E$ by a quadratic
character whose ramification at $\id{p}$ matches that of $\chi$, we
get a curve with good reduction at $\id{p}$. Note that there is no
global quadratic Grossencharacter unramified outside $\id{p}$ which
locally matches $\chi_{\id{p}}$, as the Archimedean part of all such
characters is trivial. In particular, although we can move the
ramification by twisting, there is no twist with everywhere good
reduction.

\begin{remark}
  Rational $2$-torsion is preserved under twisting, so the quadratic
  twists of $E$ have rational $2$-torsion if and only if $E$ does. The
  fields $K$ for which the curves $E$ have a $K$-rational two torsion
  point are $K=\Q(\sqrt{-d})$ for $d=2, 7$.
\label{rem:2-torsionCM}
\end{remark}

\begin{thm}
  Let $K=\Q(\sqrt{-d})$, with $d \neq 1,3$, and let $E/K$ be the base
  elliptic curve with CM by~$\Om_K$ defined above. Then all elliptic
  curves with complex multiplication over $K$ of odd prime power
  conductor are isogenous to $E$ or to the quadratic twist of $E$ by
  $\pi\sqrt{-d}$, where $\pi$ is a prime such that $\pi \equiv
  u^2\sqrt{-d} \pmod 4$ for $d\not=2$, respectively $\pi \equiv
  u^2(1+\sqrt{-d}) \pmod 4$ for $d=2$, with $u$ odd.  In particular,
  for $d=7$, the condition reads $\pi \equiv \sqrt{-7} \pmod 4$, for
  $d=2$ the condition reads $\pi \equiv \pm1 + \sqrt{-2} \pmod 4$ and
  for $d\ge11$, the condition reads $\pi \equiv w^{2k}\sqrt{-d}\pmod
  4$, for $0\le k\le2$, where $w =\frac{1+\sqrt{-d}}{2}$.
\label{theorem:CMnorootsofunity}
\end{thm}

Before giving the proof, we need an auxiliary result.

\begin{lemma}
  Let $K$ be a $2$-adic field, and $\alpha \in \Om_K$ be a $2$-adic integer
  which is a unit and is not a square. Then the extension
  $K(\sqrt{\alpha})$ is unramified if and only if there exist a unit
  $u \in \Om_K$ such that $u^2 \equiv \alpha \pmod 4$.
\label{lemma:ramificationat2}
\end{lemma}

\begin{proof}
  Let $L=K(\sqrt{\alpha})$ be the quadratic extension. The ring
  $\Om_K[\sqrt{\alpha}] \subset \Om_L$ has discriminant $4\alpha$ over
  $\Om_K$. Then the extension $\Om_L$ is unramified if and only if
  $[\Om_L:\Om_K[\sqrt{\alpha}]]=2$, if and only if
  $\frac{u+v\sqrt{\alpha}}{2} \in \Om_L$ for some $u,v \in \Om_K^*$.
  The minimal polynomial of any such element is
  $x^2 - ux + \frac{u^2-\alpha v^2}{4}$, hence the index is $2$ if and only if there
  exist units $u,v$ in $\Om_K$ such that $u^2 \equiv \alpha v^2 \pmod
  4$. Multiplying by the inverse of $v$ we get the result.
\end{proof}

\begin{proof} [of Theorem~\ref{theorem:CMnorootsofunity}]
  Since all curves with complex multiplication over $K$ are isogenous
  to a quadratic twist of $E$, we are led to determine which quadratic
  twists have bad reduction at exactly one odd prime. Any global
  character corresponds to a quadratic extension
  $K(\sqrt{\alpha})$. If the twist has good reduction at $\id{p}$,
  then $\id{p} \mid \alpha$ (and the twist attains good reduction at
  $\id{p}$), and the curve will have bad reduction at all other primes
  dividing $\alpha$. Thus $(\alpha) = \id{p} \id{q}$, for $\id{q}$ an
  odd prime.  Finally, we need to check whether the character is
  unramified at $2$, which follows from
  Lemma~\ref{lemma:ramificationat2}, and a description in each case of
  the squares modulo~$4$.

  Explicitly, for odd~$d$ we require $\alpha=\pi\sqrt{-d}$, where
  $\id{q}=(\pi)$, such that $\alpha\equiv u^2\pmod4$.  For $d=7$ the
  only odd square modulo~$4$ is~$1$.  For $d\ge11$, since $2$ is inert
  in $K$ the odd squares modulo~$4$ are the squares of the odd
  residues modulo~$2$, which are $1,w,w^2$.  For $d=2$, the twist of
  the base curve by $(1+\sqrt{-2})\sqrt{-2}$ has odd conductor
  $(1+\sqrt{-2})^2$, so we must twist by $\pi\sqrt{-2}$ where
  $\pi\equiv(1+\sqrt{-2})u^2\pmod{4}$; since the odd squares
  modulo~$4$ are $1$ and~$-1+2\sqrt{-2}$ we obtain the condition
  stated.
\end{proof}

\begin{remark}
In each case in Theorem~\ref{theorem:CMnorootsofunity}, the condition
on~$\pi$ is satisfied by one quarter of the odd residue classes
modulo~$4$.  Since $\id{q}=(\pi)$ has two generators~$\pm\pi$, our
construction gives curves of conductor~$\id{q}^2$ for half the odd
primes of~$K$.
\end{remark}

\subsection{$K=\Q(\sqrt{-1})$} The elliptic curve with complex
multiplication by $\Z[\sqrt{-1}]$ is 
\[
E:y^2=x^3+x  \qquad \text{ with label \lmfdbec{64}{a}{4}}.
\]
Its conductor over $K$ equals $\id{p}^8$, where
$\id{p}=(1+\sqrt{-1})$. In particular, $\chi_{\id{p}}$ has conductor
$\id{p}^4$ and order $4$. Note that since the automorphism group
of~$E$ is cyclic of order~$4$ we must consider \emph{quartic
  twists}. For $\alpha \in K^*$, the quartic twist of $E$ by $\alpha$ equals
\begin{equation}
  \label{eq:quartictwist}
  E_{\alpha}:y^2=x^3+\alpha x.
\end{equation}
Note that this operation does not coincide with the twist of the
L-series by a quartic character (as such L-series do not satisfy a
functional equation). Indeed, if $E$ corresponds to the automorphic
form $\chi \oplus \bar{\chi}$ (where $\chi$ is a Grossencharacter),
then $E_{\alpha}$ corresponds to the automorphic form
$\chi \psi \oplus \bar{\chi} \bar{\psi}$, where
$\psi = \kro{}{\alpha}_4$ (the quartic Legendre symbol). It is still
true that the curve $E_{\alpha}$ is isomorphic to $E$ over the
extension $K(\sqrt[4]{\alpha})$.

\begin{remark}
 All the quartic twists of $E$ have a non-trivial $K$-rational $2$-torsion point.
\end{remark}

We first need a local result about when a pure quartic extension is
unramified above~$2$.

\begin{lemma}
  Let $K=\Q_2(\sqrt{-1})$, and let $\alpha \in \Om_K$ be a
  unit. Then the extension $K(\sqrt[4]{\alpha})$ is unramified over
  $K$ if and only if $\alpha \equiv 1, 1+4\sqrt{-1} \pmod8$.
  \label{lemma:quartictwist}
\end{lemma}

\begin{proof} The extension $K(\sqrt[4]{\alpha})$ depends on $\alpha$
  up to $4$-th powers, i.e. two elements give the same extension if
  and only if they differ by a $4$-th power. By Hensel's Lemma, an odd
  element of $\Q_2(\sqrt{-1})$ is a fourth power if and only if it is
  congruent to $1$ modulo $(1+\sqrt{-1})^7$, hence the extension is
  characterized by $\alpha$ modulo $(1+\sqrt{-1})^7$. Also,
  $K(\sqrt[4]{\alpha})$ is unramified if and only if it is contained
  in the unique unramified extension of $K$ of degree~$4$, which is
  $K(\zeta_5) = K(\sqrt[4]{1+4\sqrt{-1}})$, as can be easily
  checked. Thus, for $K(\sqrt[4]{\alpha})$ to be unramified, $\alpha$
  must be congruent to a power of $1+4\sqrt{-1}$, i.e.  $\alpha \equiv
  1, 1+4\sqrt{-1}, 9,\text{ or } 9+4\sqrt{-1} \pmod{(1+\sqrt{-1})^7}$;
  this simplifies to $\alpha\equiv 1, 1+4\sqrt{-1}\pmod{8}$ as stated.
\end{proof}

\begin{thm}
  Let $K=\Q(\sqrt{-1})$ and let $E/K$ be the elliptic curve \lmfdbec{64}{a}{4}. Then
  all elliptic curves with complex multiplication over $K$ of odd
  prime power conductor are isogenous to the quartic twist of $E$ by
  $\pi$, where $\pi$ is a prime power such that
  $\pi \equiv -1 \pm 2\sqrt{-1} \pmod 8$.
\label{theorem:CMGaussian}
\end{thm}

\begin{proof}[of Theorem~\ref{theorem:CMGaussian}] Let
  $\pi = -1+2\sqrt{-1}$. One may check that the quartic twist of~$E$
  by $\pi$ has good reduction at $2$ and bad additive reduction at
  $(-1+2\sqrt{-1})$. Then any quartic twist of $E$ of odd conductor is
  a twist of $E_{-1+2\sqrt{-1}}$ by a quartic character of odd
  conductor, which by Lemma~\ref{lemma:quartictwist} correspond to
  elements which are congruent to $1$ or~$1+4\sqrt{-1}\pmod8$.
  Multiplying by $-1+2\sqrt{-1}$ gives the classes~$-1\pm2\sqrt{-1}\pmod8$
  as stated.
\end{proof}

\begin{remark}
  Of all odd primes~$\id{q}$ of $K$, one quarter have a
  generator~$\pi$ satisfying the condition in
  Theorem~\ref{theorem:CMGaussian}. Hence our construction gives
  elliptic curves of conductor~$\id{q}^2$ for one quarter of all
  primes~$\id{q}$ of~$K$.

\end{remark}

\subsection{$K=\Q(\sqrt{-3})$}
The elliptic curve with complex multiplication by
$\Z[\frac{1+\sqrt{-3}}{2}]$ is the curve \lmfdbec{27}{a}{4}, with (non-minimal)
equation $$E:y^2=x^3+16.$$ Its conductor over $K$ is $\id{p}^4$, where
$\id{p}=(\sqrt{-3})$. In particular, $\chi_{\id{p}}$ has conductor
$\id{p}^2=(3)$ and order $6$ (note that $(\Om_K/3)^\times \simeq
\Z/6\Z$).  Since $E$ has automorphism group of order~$6$, we must
consider sextic twists, where the sextic twist by $\alpha \in K^*$ of
the previous model is
\begin{equation}
  \label{eq:sextictwist}
  E_{\alpha}:y^2=x^3+16\alpha.
\end{equation}
Similar considerations as for the quartic twists apply. In particular,
the curve $E_{\alpha}$ is isomorphic to $E$ over the extension
$K(\sqrt[6]{\alpha})$;
such twists are needed to cancel the CM character $\chi_{\id{p}}$.

As before, we need local results, now at both~$2$ and~$3$:

\begin{lemma}
  Let $K=\Q(\sqrt{-3})$, let $w=(1+\sqrt{-3})/2\in K$ be a $6$th root of
  unity, and let $\alpha \in \Om_K$ be a $2$-adic and $3$-adic
  unit. Then the extension $K(\sqrt[6]{\alpha})/K$ is unramified over
  both $2$ and~$3$ if and only if
  \begin{enumerate}
  \item $\alpha \equiv 1, w^2$, or $w^4\pmod 4$ (equivalently, $\alpha$ is
    congruent to a square modulo~$4$); and
  \item $\alpha \equiv \pm 1 \pmod {\sqrt{-3}^3}$ (equivalently,
    $\alpha$ is congruent to a cube modulo~$\sqrt{-3}^3$).
  \end{enumerate}
  \label{lemma:sextictwist}
\end{lemma}

\begin{proof} Since $K(\sqrt[6]{\alpha}) =
  K(\sqrt{\alpha},\sqrt[3]{\alpha})$ we require both
  $K(\sqrt{\alpha})/K$ and $K(\sqrt[3]{\alpha})/K$ to be unramified.
  The first is certainly unramified over~$3$, and over~$2$ we may
  apply Lemma~\ref{lemma:ramificationat2} to obtain the first
  condition stated.

  Similarly, $K(\sqrt[3]{\alpha})/K$ is always unramified over~$2$, so
  we need the condition for it to be unramified also over~$3$.  By
  Hensel's Lemma, a unit of $\Q_3(\sqrt{-3})$ is a cube if and only if
  it is congruent to $\pm1$ modulo~$9$, hence the extension is
  characterized by $\alpha$ modulo~$9$.  We may check that
  $K(\sqrt[3]{\alpha_1})$ is unramified, for $\alpha_1=2+3w$.  Hence
  $K(\sqrt[3]{\alpha})/K$ is unramified at~$3$ if and only if
  $\alpha\equiv\pm1, \pm\alpha_1, \pm\alpha_1^2\pmod9$, which is if
  and only if $\alpha\equiv\pm1\pmod{\sqrt{-3}^3}$.
\end{proof}

\begin{thm}
  Let $K=\Q(\sqrt{-3})$ let $E/K$ be the elliptic curve \lmfdbec{27}{a}{4}. Then
  all elliptic curves with complex multiplication over $K$ of odd
  prime power conductor are isogenous to $E$ or to the sextic twist
  $E_{\alpha}$ of $E$ by $\alpha=\sqrt{-3}^3\, \pi$, where $\pi$ is a
  prime power such that:
  \begin{enumerate}
  \item $\pi \equiv \sqrt{-3}, \sqrt{-3}w^2$, or $\sqrt{-3}w^4\pmod 4$, and
  \item $\pi \equiv \pm 4 \pmod {\sqrt{-3}^3}$,
  \end{enumerate}
   where $w=(1+\sqrt{-3})/2$ is a 6th root of unity.
\label{theorem:CMEisenstein}
\end{thm}

\begin{proof}
  $E$ itself has good reduction except at $\sqrt{-3}$; by
  Lemma~\ref{lemma:sextictwist}, the sextic twist $E_{\alpha}$ will
  also have good reduction at~$2$ provided that $\alpha$ is an odd
  square modulo~$4$, equivalently $\alpha\equiv1,w^2,w^4\pmod4$.
  Hence the first condition on~$\pi$ ensures that $E_{\alpha}$ has
  good reduction except at $\pi$ and (possibly) at~$\sqrt{-3}$.

  The twist $E_{\alpha_1}$ with $\alpha_1=\sqrt{-3}^3\cdot 4$ has good
  reduction at $\sqrt{-3}$.  Hence by Lemma~\ref{lemma:sextictwist},
  $E_{\alpha}$ has good reduction at~$\sqrt{-3}$ if $\alpha/\alpha_1$
  is a cube modulo~$\sqrt{-3}^3$, or equivalently
  $\alpha/\alpha_1\equiv\pm1\pmod{\sqrt{-3}^3}$.  This is ensured by
  the second condition, since $\pi/4=\alpha/\alpha_1$.
\end{proof}

\begin{remark}
  The curves in the previous family never have a rational $2$-torsion
  point, since $2\alpha$ is not a cube.
\end{remark}

\begin{remark}
  Of all primes of $K$ other than~$(2)$ and $(\sqrt{-3})$, half have a
  generator~$\pi$ satisfying the $2$-adic condition in
  Theorem~\ref{theorem:CMEisenstein}, and one third have a generator
  satisfying the $3$-adic condition.  Hence our construction gives
  elliptic curves of conductor~$\id{p}^2$ for one sixth of all
  primes~$\id{p}$ of~$K$.
\end{remark}

\subsubsection{Curves with complex multiplication by $K$ over an
  imaginary quadratic field $L$.} A natural question is what happens if
we consider a curve $E$ with complex multiplication by an order in
$K$, over a possibly different imaginary quadratic field $L$: are there twists of $E$
with good reduction at primes dividing $\cond(E)$?

The proofs of the previous results are of a local nature, hence if $L$
has the same completion at a prime dividing $\cond(E)$ as $K$ we are
in exactly the same situation.

\begin{thm}
  Let $E/\Q$ be an elliptic curve of conductor $p^r$ with complex
  multiplication by an order in $K$. Let $L=\Q(\sqrt{-t})$ be an
  imaginary quadratic field different from $K$ and $\id{p}$ a prime
  ideal of $L$ dividing $p$. If the completion of $L$ at $\id{p}$ is
  isomorphic to the completion of $K$ at the prime dividing $p$, then
  there exists $\alpha \in L$ such that:
  \begin{enumerate}
  \item if $K=\Q(\sqrt{-d})$, $d \neq 1,3$, then the quadratic twist of
    $E/L$ by $\sqrt{-t}\, \alpha$ has good reduction at $\id{p}$.
    
  \item if $K=\Q(\sqrt{-1})$, let $\sqrt{-1}$ denote an element of $L$
    whose square is congruent to $-1$ modulo $8$. Then the quartic
    twist of the curve \lmfdbec{64}{a}{4} by $\alpha$ has good reduction at $2$ for
    $\alpha \equiv -1 \pm 2\sqrt{-1} \pmod 8$.
     
  \item if $K=\Q(\sqrt{-3})$, let $\sqrt{-3}$ denote an element of $L$
    whose square is congruent to $-3$ modulo $9$. Then the sextic
    twist of the curve \lmfdbec{27}{a}{4} by $\alpha$ has good reduction at $3$
    for $\alpha \equiv \pm 4\sqrt{-3}^3 \pmod {\sqrt{-3}^3}$.
  \end{enumerate}
  On the other hand, if the completions are not isomorphic, no such
  twist exists.
  \label{thm:otherfields}
\end{thm}
\begin{proof}
  The proof of the first facts mimics that of Theorems
  \ref{theorem:CMnorootsofunity}, \ref{theorem:CMGaussian} and
  \ref{theorem:CMEisenstein}, as the completions being isomorphic
  implies that the local reduction types are the same. Suppose that
  the local completions are not isomorphic. Consider the Weil
  representation at $p$ attached to our elliptic curve $E/\Q$ (recall
  that CM elliptic curves have no monodromy, hence we do not need to
  consider the whole Weil-Deligne representation). Then it is easy to
  check that the image of inertia at $p$ equals:
  \begin{enumerate}
  \item  a cyclic group of order $4$ if $d \neq 1, 3$,
  \item the dihedral group of order $8$ if $d = 1$,
  \item the dihedral group of order 12 if $d=3$.
  \end{enumerate}
  Recall that the curve $E/L$ will have good reduction at $\id{p}$ if
  the restriction of the Weil representation to the inertia subgroup of $L$
  at $\id{p}$ is trivial.

  In the first case, since the completion of $L$ at $\id{p}$ is not
  isomorphic to the completion of $K$ at $\id{p}$, the restriction
  to the inertia subgroup of $L$ at $\id{p}$ still has order
  $4$. But since $d \neq 1$, we cannot take quartic twists, hence
  we cannot cancel the ramification, and all such curves $E/L$ will
  have bad reduction at $\id{p}$. In the other two cases, the image of
  the inertia subgroup of $\Q$ at $p$ is not abelian, and the unique
  quadratic extension whose restriction becomes abelian is
  $K_{\id{p}}$, hence by twisting we cannot kill the ramification for
    any other quadratic extension of $\Q_p$.
\end{proof}

\begin{coro}
  If $E/\Q$ is an elliptic curve with complex multiplication by an
  order in $K$ and $L$ is an imaginary quadratic field of class
  number $1$ different from $K$, then $E$ is the unique curve in its
  family of appropriate twists (quadratic, quartic or sextic) which has
  prime power conductor.
  \label{coro:differentfields}
\end{coro}
\begin{proof} Let $D=\{1,2,3,7,11,19,43,67,163\}$. The field
  $\Q(\sqrt{-d})$ for each odd $d \in D$ has a different ramification
  set, hence we cannot get good reduction at $d$ by Theorem
  \ref{thm:otherfields}. For $d$ equal to $1$ or $2$ the completions
  are also different, hence we cannot get odd conductor from the curve
  with CM by an order of $\Q(\sqrt{-1})$ over $\Q(\sqrt{-2})$, or vice
  versa.
\end{proof}

\section{Prime power conductor curves with rational odd torsion}
\label{smallimage}
Recall the following result of \cite{Kenku} and \cite{Kamienny}:

\begin{thm}
  Let $K$ be a quadratic field, and $E/K$ be an elliptic curve. Then
  $E(K)_{tors}$ is isomorphic to one of the following groups:
  \begin{enumerate}
  \item $\ZZ/N$, with $ 1 \le N \le 18$ but $N \neq 17$.
  \item $\ZZ/2N \times \ZZ/2$ with $1 \le N \le 6$.
  \item $\ZZ/4 \times \ZZ/4$.
  \item $\ZZ/3 \times \ZZ/3N$ with $N=1,2$.
  \end{enumerate}
\end{thm}

In particular the primes dividing the order of the torsion subgroup
are $2, 3, 5, 7, 11$ and $13$. Let $\id{q} \mid 2$ be a prime and
$E$ be an elliptic curve of odd conductor. By Hasse's bound
\[
\#E(\FF_{\id{q}})=|\normid{q}+1-a_E(\id{q}) | \le \normid{q}+1 + 2 \sqrt{\normid{q}} < 11.
\]
In particular a curve of odd prime power conductor over $K$ can only
have a torsion point of odd prime order $\ell$ for
$\ell \in \{3,5,7\}$. 

While studying the possible torsion of an elliptic curve over $K$, the
case $\ell =3$ and $K=\Q(\sqrt{-3})$ is quite different from the
others. The reason is that since $K$ contains the sixth roots of
unity, the determinant of the Galois representation acting on
$3$-torsion points is trivial. The main results of the present section
are Theorem \ref{thm:3-torsionspecialcase}, which implies that curves
over $\Q(\sqrt{-3})$ of odd prime conductor and with a rational
$3$-torsion point, have (up to $3$-isogeny) discriminant valuation not
divisible by~$3$;
and a
list of all elliptic curves of prime power conductor and a point
of order $\ell \in \{3, 5, 7\}$ is given for all the other fields. The complete list
(omitting Galois conjugates) is given in Table~\ref{table:oddtorsioncurves}, whose completeness
will be proved in this section, in Theorems
\ref{thm:3-torsion},~\ref{thm:5-torsion}, and~\ref{thm:7-torsion}.

\begin{remark}
  Besides curves over $\Q(\sqrt{-3})$ with a rational $3$-torsion
  point, the only curves of prime conductor over imaginary quadratic
  fields with a rational $\ell$-torsion point are those over
  $\Q(\sqrt{-1})$ with a $3$-torsion point, and base-changes of
  elliptic curves defined over~$\Q$.
\end{remark}

Let $p$ be an odd prime, $K/\Q$ be a quadratic extension and $E/K$ be
an elliptic curve with a global minimal model. If $P \in E(K)$ has order $p$, then by
\cite[VII, Theorem 3.4]{Silverman} $P$ has algebraic integer coordinates in
the minimal model, except when
$p=3$ and $K/\Q$ is ramified at $3$ where, if $\id{p}_3$ denotes the
prime dividing $3$, the case $v_{\id{p}_3}(x(P),y(P))= (-2,-3)$ might
occur.

\begin{table}
\scalebox{.9}
{
\begin{tabular}{||r|r|r|r|r||}
\hline
Label & Weierstrass Coefficients & $\ell$ & $\Disc(E)$ & Field\\
\hline

\lmfdbec{19}{a}{2} & $[0, 1, 1, -9, -15]$ & $3$ & $-19^3$ & $1, 7, 11, 19, 43, 163$\\
\lmfdbec{19}{a}{3} & $[0, 1, 1, 1, 0]$ & $3$ & $-19$ & $1, 7, 11, 19, 43, 163$\\
  \lmfdbec{27}{a}{2} & $[0, 0, 1, -30, 63]$ & $3$ & $-3^5$ & $1, 7, 19, 43, 67, 163$\\
  \lmfdbec{27}{a}{3} & $[0, 0, 1, 0, -7]$ & $3$ & $-3^9$ & $1, 7, 19, 43, 67, 163$\\
\lmfdbec{27}{a}{4} & $[0, 0, 1, 0, 0]$ & $3$ & $-3^3$ & $1, 7, 19, 43, 67, 163$\\
\lmfdbec{37}{b}{2} & $[0, 1, 1, -23, -50]$ & $3$ & $37^3$ & $2, 19, 43, 163$\\
\lmfdbec{37}{b}{3} & $[0, 1, 1, -3, 1]$ & $3$& $37$ &  $2, 19, 43, 163$\\
  \lmfdbec{243}{a}{2} & $[0, 0, 1, 0, 20]$ & $3$& $-3^{11}$ &  $1, 7, 19, 43, 67, 163$\\
  \lmfdbec{243}{b}{2} & $[0, 0, 1, 0, 2]$ & $3$& $-3^7$ &  $1, 7, 19, 43, 67, 163$\\
  \lmfdbecnf{2.0.4.1}{757.1}{a}{1} & $[1+i,-1+i,1,-15+5i,-17+6i]$ & $3$ & $ (-26+9i)^3$ & $1$\\
\lmfdbecnf{2.0.4.1}{757.1}{a}{2} & $[1+i,1+i,1,2i,i]$ & $3$ & $ -26+9i$ & $1$\\
\lmfdbecnf{2.0.8.1}{9.1}{CMa}{1} & $[\sqrt{-2}, 1-\sqrt{-2}, 1, -1, 0]$ & $3$ & $({-1-\sqrt{-2}})^6$ & $2$\\
\lmfdbecnf{2.0.11.1}{9.3}{CMa}{1} & $[0, \frac{1-\sqrt{-11}}{2}, 1, \frac{-5-\sqrt{-11}}{2}, -2]$ & $3$ & $({\frac{-1+\sqrt{-11}}{2}})^6$ & $11$\\
\hline
\lmfdbec{11}{a}{2} & $[0, -1, 1, -10, -20]$ & $5$ & $-11^5$ & $1, 3, 11, 67, 163$\\
\lmfdbec{11}{a}{3} & $[0, -1, 1, 0, 0]$  & $5$ & $-11$ & $1, 3, 11, 67, 163$\\
\lmfdbecnf{2.0.4.1}{25.3}{CMa}{1} & $[1+i,i,i,0,0]$ & $5$ & $(2i+1)^3$ & $1$\\
\hline
\lmfdbecnf{2.0.3.1}{49.3}{CMa}{1} & $[0,\frac{-3+\sqrt{-3}}{2},\frac{1+\sqrt{-3}}{2},\frac{1-\sqrt{-3}}{2},0]$&$7$& $(\frac{-1-3\sqrt{-3}}{2})^2$ & $3$ \\
\hline
\end{tabular}}
\caption{Prime power conductor curves with an $\ell$-torsion point
  ($\ell$ odd)}
\label{table:oddtorsioncurves}
\end{table}

\begin{thm}
  Let $K=\Q(\sqrt{-3})$, and let $E/K$ be a curve with a point $P$ of
  order $3$ and prime power conductor $\id{p}^r$ and let $\tilde{E}$
  be the $3$-isogenous curve obtained by taking the quotient of $E$ by
  the group generated by $P$. Then the valuations at $\id{p}$ of
  $\Disc(E)$ and $\Disc(\tilde{E})$ are not both divisible by $3$,
  unless $E$ is one of the CM curves
  \lmfdbecnf{2.0.3.1}{729.1}{CMa}{1} or
  \lmfdbecnf{2.0.3.1}{729.1}{CMb}{1}.
\label{thm:3-torsionspecialcase}
\end{thm}
\begin{proof}
  Suppose that $(\Disc(E))=\id{p}^{3r}$.  
  Let $P$ denote the point of order $3$ in $E(K)$. If $P$ has integral
  coordinates, we use the parametrization of elliptic curves with a
  rational point of order~$3$ given by Kubert in \cite[Table
  1]{Kubert}: such curves have a minimal model of the form
\begin{equation}
E:y^2+a_1 xy+ a_3 y= x^3,
\label{eq3-torsion1}
\end{equation}
where $a_i$ are algebraic integers, $P=(0,0)$ has order~$3$, with
discriminant
\begin{equation}
\Disc(E) = a_3^3 (a_1^3 -27 a_3).
\label{eq:discriminant3torsion1}
\end{equation}
If $P$ does not have integral  coordinates, then we take the minimal
equation to one of the form (\ref{eq3-torsion1}), with $v_{\id{p}}(a_1) \ge 0$ and $v_{\id{p}}(a_3)=-3$.

Note that over a field containing the
  $3$-rd roots of unity the cyclotomic character modulo
  $3$ is trivial so the representation of the Galois group acting on
  $E[3]$ has image in
  $\SL(2,3)$.  Then if it is reducible, with upper triangular
  matrices, the diagonal entries are both $+1$ or both
  $-1$.  So if the curve~$E$ has a rational point of order
  $3$, so does the isogenous curve
  $\tilde{E}$. We find that $\tilde{E}$ has equation $y^2+a_1 xy+ a_3
  y= x^3-5a_1a_3x-a_1^3a_3-7a_3^2$, and discriminant
    \begin{equation}
    \label{eq:discEt}
    \Disc(\tilde{E})=a_3(a_1^3-27a_3)^3.
  \end{equation}
  Suppose that both $\Disc(E)$ and $\Disc(\tilde{E})$ generate ideals which are cubes. Then
  \begin{align*}
    a_1^3-27a_3 & = u \alpha^3\\
    a_3& =  v \beta^3,
  \end{align*}
  for $u,v$ units and $\alpha, \beta \in
  K^\times$; in fact, $\alpha,\sqrt{-3}\beta\in\Om_K$.  In particular,
  $(a_1: -\alpha: -3\beta)$ is a $K$-rational point on the cubic curve
    \begin{equation}
x^3+uy^3+vz^3=0,     \label{eq:fermatcubics}
    \end{equation}
a twist of the Fermat cubic.  By Lemma~\ref{lem:fermatcubics} below,
all $K$-rational points~$(x:y:z)$ on all curves of the
form~(\ref{eq:fermatcubics}) either satisfy $xyz=0$, or (after scaling
so that $x,y,z\in\Om_K$ are coprime) that $x,y,z$ are all units.

Since $\alpha\beta\not=0$ the first case is possible only when
$a_1=0$.  Then $a_3$ is a unit times a cube, so by minimality is a
unit: this leads to the three isomorphism classes of curves with
conductor~$(9)$ or $(27)$. The first of these is
\lmfdbecnf{2.0.3.1}{81.1}{CMa}{1}, whose discriminant valuation is~$6$
but has a $3$-isogenous curve with discriminant valuation~$10$; the
other two, which are given in the statement of the theorem, are each
isomorphic to their $3$-isogenous curves (the isogeny being an
endomorphism) and have discriminant valuation~$6$.

In case none of the coordinates is zero, we consider separately the
finitely many cases where $\beta$ is integral or has valuation~$-1$,
and find that there are no more solutions.
\end{proof}

\begin{lemma}\label{lem:fermatcubics}
Let $K=\Q(\sqrt{-3})$ and $u,v\in\Om_K^*$.  The cubic
curve~(\ref{eq:fermatcubics}) has either $3$ or $9$ rational points,
which either lie on one of the lines $x=0$, $y=0$ or $z=0$, or have
projective coordinates which are all units.
\end{lemma}
\begin{proof}
After permuting the coordinates, scaling by units and absorbing cubes,
there are only three essentially different equations, those with
$(u,v)=(1,1)$, $(1,\zeta)$, and $(\zeta,\zeta^2)$ where $\zeta\in K$
is a $6$th root of unity.  When $u=1$, it is well-known that all
points have one zero coordinate (see \cite[Proposition
  17.8.1]{Ireland-Rosen}).  There are $9$ such points (all the flexes)
when $v=1$, and $3$ when $u=\zeta$.  The curve with
$(u,v)=(\zeta,\zeta^2)$ is isomorphic to the one with $(u,v)=(1,1)$,
since the isomorphism class depends only on $uv$ modulo cubes, and
hence also has $9$ points; these are $(\zeta^{2k}:\zeta^{2l+1}:1)$ for
$k,l\in\{0,1,2\}$.
\end{proof}

\begin{thm}
  Let $K=\Q(\sqrt{-d})$ be an imaginary quadratic field of class
  number $1$, $K \neq \Q(\sqrt{-3})$ and $E/K$ an elliptic curve of
  odd prime power conductor with a point of order $3$. Then $E$ is
  isomorphic to one of the following:
  \begin{enumerate}
  \item the curve \lmfdbec{19}{a}{2} or \lmfdbec{19}{a}{3} over
    $\Q(\sqrt{-d})$  for $d=1, 7, 11, 19, 43, 163$;
  \item the curve \lmfdbec{27}{a}{2} or \lmfdbec{27}{a}{3}  or \lmfdbec{27}{a}{4}  over
    $\Q(\sqrt{-d})$ for $d=1, 7, 19, 43, 67, 163$;
  \item the curve \lmfdbec{37}{b}{2} or \lmfdbec{37}{b}{3}  over
    $\Q(\sqrt{-d})$ for $d=2, 19, 43, 163$;
  \item the curve \lmfdbec{243}{a}{2} or \lmfdbec{243}{b}{2} over
    $\Q(\sqrt{-d})$ for $d=1, 7, 19, 43, 67, 163$;
  \item $K=\Q(\sqrt{-1})$ and $E$ is one of the $3$-isogenous curves
\[
  y^2 + (1+i) xy + i y = x^3 + (-1+i) x^2 + (-14-8i) x + (-10-20i),
\]
with label \lmfdbecnf{2.0.4.1}{757.1}{a}{1},
\[    y^2 + (1+i) xy + i y = x^3 + (1+i) x^2 + (1-i) x + 2,
  \] with label \lmfdbecnf{2.0.4.1}{757.1}{a}{2},
or their Galois conjugates with labels
\lmfdbecnf{2.0.4.1}{757.2}{a}{1} and \lmfdbecnf{2.0.4.1}{757.2}{a}{2}.
\item $K=\Q(\sqrt{-2})$ and $E$ is the curve
\[
  y^2 + \sqrt{-2} xy + y = x^3 + (1-\sqrt{-2}) x^2 -x,
  \]
    \text{with label \lmfdbecnf{2.0.8.1}{9.1}{CMa}{1}},
    or its Galois conjugate with label
    \lmfdbecnf{2.0.8.1}{9.3}{CMa}{1};
\item $K=\Q(\sqrt{-11})$   and $E$ is the curve
\[
    y^2 + y = x^3 + \frac{1-\sqrt{-11}}{2} x^2 +
    \frac{-5-\sqrt{-11}}{2} x - 2,
  \]with label \lmfdbecnf{2.0.11.1}{9.3}{CMa}{1}, or its Galois
  conjugate with label
  \lmfdbecnf{2.0.11.1}{9.1}{CMa}{1}.
  \end{enumerate}
\label{thm:3-torsion}
\end{thm}

\begin{proof}
As in Theorem~\ref{thm:3-torsionspecialcase}, we use the
parametrization of elliptic curves with a rational point of order~$3$
given by Kubert in \cite[Table 1]{Kubert}: such curves have a model of
the form
\begin{equation}
E:y^2+a_1 xy+ a_3 y= x^3,
\label{eq3-torsion}
\end{equation}
where $P=(0,0)$ has order~$3$, with discriminant
\begin{equation}
\Disc(E) = a_3^3 (a_1^3 -27 a_3).
\label{eq:discriminant3torsion}
\end{equation}
By scaling, we can choose a model of this form such that for all
primes~$\id{q}$ either $\id{q}\nmid a_1$ or $\id{q}^3\nmid a_3$. Then
the model is minimal at all primes, as we now show.  To be non-minimal
at a prime $\q$ implies that $\id{q}^6\mid c_6$ and $\id{q}^{12}\mid
\Disc(E)$, where $c_6$ is the usual invariant of the model.  The ideal
generated by $c_6$ and $\Disc(E)$ in the ring $\Z[a_1,a_3]$ contains
both $a_1^{15}$ and $3^3a_3^5$, so $\id{q}\mid a_1$ and $\id{q}\mid
a_3$: in case $\id{q}\mid3$, we need the fact that $3$ is not ramified
in~$K$.  By minimality, $v_{\id{q}}(a_3)\in\{1,2\}$; this implies
$v_{\id{q}}(\Disc(E))\le11$, contradiction.

Let $\id{p}=(\pi)$ be the unique prime dividing $\Disc(E)$. As above,
we can assume that either $\id{p}\nmid a_1$ or $\id{p}^3\nmid a_3$.
We can also scale by units, replacing $(a_1,a_3)$ by $(ua_1,u^3a_3)$,
and we note that for the fields under consideration every unit is a
cube.  Our strategy is to prove that $(a_1,a_3)$ lies in a small
finite set, and then systematically search through all possible
values.
\begin{enumerate}
\item If $a_3$ is a unit, we may assume by scaling that $a_3=1$. Then
  we can factor (\ref{eq:discriminant3torsion}) as
\[
\Disc(E) = a_1^3 -27 = (a_1 -3) (a_1^2 + 3  a_1 +9 ).
\]
We consider the following cases:
\begin{enumerate}
\item If $\id{p}\nmid3$ then the factors are coprime, so one is a
  unit. If $u=a_1-3\in\Om_K^*$ we get four solutions
  $(a_1,a_3)=(4,1)$, $(2,1)$, $(3\pm\sqrt{-1},1)$ which give the
  curves \lmfdbec{37}{b}{3}, \lmfdbec{19}{a}{3},
  \lmfdbecnf{2.0.4.1}{757.1}{a}{2} and
    \lmfdbecnf{2.0.4.1}{757.2}{a}{2}.  If $a_1^2 + 3 a_1
      +9 = u \in\Om_K^*$ then $(2a_1+3)^2 = 4u-27 \in
      \{-23,-31,-27\pm4\sqrt{-1}\}$ which has no solution in~$K$.
\item If $\id{p} \mid 3$ then $\id{p} \mid a_1$.  If
  $v_{\id{p}}(a_1)=1$, we write $a_1=\pi b$.  In the inert case,
  $\pi=3$ and $\Disc(E)=27(b^3-1)=27(b-1)(b^2+b+1)$.  If either factor
  is a unit one finds no solutions, otherwise both are divisible
  by~$3$, so write $b=1+3c$; now the second factor is $3(1+3c+3c^2)$
  so $1+3c+3c^2$ is a unit.  Only $u=1$ gives a solution: $c=-1$,
  $b=-2$ and $a_1=-6$.  The pair $(a_1,a_3)=(-6,1)$ yields the curve
  \lmfdbec{27}{a}{4}.  In the split case we get $\Disc(E) =
  \pi^3(b^3-\overline{\pi}^3) = \pi^3(b-\overline{\pi})
  (b^2+b\overline{\pi}+\overline{\pi}^2)$.  Elementary computations
  reveal two solutions: $b=\sqrt{-2}$ giving
  $(a_1,a_3)=(\sqrt{-2}-2,1)$ and the curve
  \lmfdbecnf{2.0.8.1}{9.1}{CMa}{1}; and $b=-2$ giving
  $(a_1,a_3)=(-1-\sqrt{-11},1)$ and the curve
  \lmfdbecnf{2.0.11.1}{9.1}{CMa}{1}; together with their Galois
  conjugates.

  If $a_1=0$ or $v_{\id{p}}(a_1)\ge2$ then $v_{\id{p}}(\Disc(E))=3$.  This gives
  the equation $a_1^3=27+\pi^3u$ with $u$ a unit.  When $3$ is inert
  we have $\pi=3$ and $(a_1/3)^3=1+u\in\{2,0,1\pm\sqrt{-1}\}$, giving
  just one solution $(a_1,a_3)=(0,1)$ and the curve
  \lmfdbec{27}{a}{4}.  When $3=\pi\overline{\pi}$ we have
  $(a_1/\pi)^3=\overline{\pi}^3\pm1$ but this is not a cube.  (Here,
  $\pi=1\pm\sqrt{-2}$ or $\pi=(1\pm\sqrt{-11})/2$.)
\end{enumerate}

\item Now suppose that $a_3$ is not a unit. If $a_1=0$, then $3$ is
  inert in $K$ and $a_3 = 3^j$, with $j \in \{1,2\}$ by the minimality
  of the model, corresponding to the curves \lmfdbec{243}{b}{2} and
  \lmfdbec{243}{a}{2} respectively. Suppose that $a_1 \neq 0$.
  \begin{enumerate}
    \item If $\id{p} \mid a_1$, the minimality of the model implies
      that $v_{\id{p}}(a_3) < 3$; scaling by units we can assume that
      $a_3=\pi^j$ with $j\in\{1,2\}$.

  If $\id{p}\nmid3$ then $v_{\id{p}}(a_1^3-27a_3)=v_{\id{p}}(a_3)=j$
  so $a_1^3=\pi^j(27+v)$ for some unit~$v$, but none of these
  expressions is a cube.  Hence $\id{p}\mid3$.

  If $v_{\id{p}}(a_1)=1$, we have $v_{\id{p}}(a_1^3-27a_3)=3$ so
  $a_1^3=27a_3+\pi^3u$ with $u$ a unit.  In the inert cases
  $(a_1/3)^3=a_3+u$, whose only solution is $(a_1,a_3)=(6,9)$ which
  yields the curve \lmfdbec{27}{a}{3}.  In the split cases we find no
  solutions.

  If $v_{\id{p}}(a_1)\ge2$, we have
  $v_{\id{p}}(a_1^3-27a_3)=v_{\id{p}}(27a_3)=3+j$ so
  $0\not=a_1^3=27a_3+\pi^{3+j}u=\pi^j(27+\pi^3u)$ with $u$ a unit. This has
  no solutions.

\item If $\id{p} \nmid a_1$, then $a_1^3-27a_3$ is a unit and we can
  scale so the unit is~$1$, so $a_1^3=1+27a_3$.

  If $\id{p}\mid3$ then from $27a_3=a_1^3-1\equiv(a_1-1)^3\pmod3$ we
  have $\id{p}\mid a_1=1$ and either $a_1-1$ or $a_1^2+a_1+1$ has
  valuation~$1$, but no solutions arise.

  Hence $\id{p}\nmid3$. Now $27a_3 = a_1^3-1 = (a_1-1)(a_1^2+a_1+1)$,
  where the gcd of the factors divides~$3$.  One of the factors is
  coprime to~$\id{p}$, so divides~$27$, and both factors are divisible
  by the prime or primes dividing~$3$, so $3\mid(a_1-1)$.  Suppose
  that we are in the case that $a_1-1\mid27$.  In case $3$ is inert,
  write $a_1-1=3^ku$ with $u$ a unit and $k\in\{1,2,3\}$.  The only
  cases where $a_1^3\equiv1\pmod{27}$ are $(a_1,a_3)=(10,37)$ and
  $(-8,-19)$, giving the curves \lmfdbec{37}{b}{2} and
  \lmfdbec{19}{a}{2} respectively.  In case $3$ splits as
  $3=\omega\overline{\omega}$ we have
  $a_1=1\pm\omega^k\overline{\omega}^l$ with $k,l\in\{1,2,3\}$; the
  only solutions are those with have $k=l=2$ which have already been
  seen.  Secondly, if $a_1^2+a_1+1=d\mid27$ then the quadratic
  in~$a_1$ has discriminant $4d-3$ which must be a square; enumeration
  of cases shows no solutions.

  \end{enumerate}
  \end{enumerate}

In summary we find that the only solutions~$(a_1,a_3)$, up to scaling
by units and Galois conjugates, are $(a_1,1)$ for $a_1\in\{0, 2, 4,
-6, 3+\sqrt{-1}, -2+\sqrt{-2}, -1-\sqrt{-11}\}$, and $(0,3)$, $(0,9)$,
$(6,9),(10,37)$, and $(-8,-19)$.
\end{proof}

\begin{remark}
  It is an interesting question to determine whether there are
  infinitely many curves over $\Q(\sqrt{-3})$ with a point of order
  $3$ and prime conductor. Based on numerical evidence, it seems quite
  plausible that this is indeed the case, but we did not focus on this particular
  problem.
\end{remark}

\begin{thm}
  Let $K=\Q(\sqrt{-d})$ be an imaginary quadratic field of class number $1$ and $E/K$
  an elliptic curve of odd prime power conductor with a point of order
  $5$. Then either $E$ is isomorphic to the curve \lmfdbec{11}{a}{2} or
  \lmfdbec{11}{a}{3} for $d=1, 3, 11, 67, 163$ or $K=\Q(\sqrt{-1})$ and $E$ is the curve:
  \begin{eqnarray}
    \label{eq:curves5torsion}
    y^2 + (i+1) xy+iy = x^3 + ix^2,\qquad&&\text{with label \lmfdbecnf{2.0.4.1}{25.3}{CMa}{1}},
  \end{eqnarray}
or its Galois conjugate with label \lmfdbecnf{2.0.4.1}{25.1}{CMa}{1}.
\label{thm:5-torsion}
\end{thm}
Note that the curve \lmfdbecnf{2.0.4.1}{25.3}{CMa}{1} has CM by
$\ZZ[i]$; in particular, the conductor valuation is $2$.

\begin{proof}
Again we use Kubert's parametrization from \cite[Table 1]{Kubert}: such curves have
a model of the form
\[
y^2+(1-d)xy-dy = x^3-dx^2,
\]
with $d \in K$. An integral model is then given by
\[
E_{a,b}:y^2+(b-a)xy-ab^2y = x^3 -abx^2, \qquad \text{ with }\gcd(a,b)=1.
\]
This model has discriminant $\Disc(E_{a,b}) = a^5b^5(a^2-11ab-b^2)$,
and $c_4$-invariant $a^{4} - 12 a^{3} b + 14 a^{2} b^{2} + 12 a b^{3}
+ b^{4}$.  In the polynomial ring $\Z[a,b]$ the ideal these generate
contains $5a^{15}$ and~$5b^{15}$, so they are coprime away from~$5$.
Hence at all primes except possibly those dividing~$5$ the model
$E_{a,b}$ is minimal, and has multiplicative reduction.

Let $\id{p}$ be a prime above~$5$ and suppose that $\id{p}$ divides
both $\Disc(E_{a,b})$ and $c_4(E_{a,b})$.  Then $\Disc(E_{a,b}) \equiv
a^5b^5(a+2b)^2\pmod{\id{p}}$ and
$c_4(E_{a,b})\equiv(a+2b)^4\pmod{\id{p}}$, so $\id{p}\mid(a+2b)$.
Writing $a=-2b+c$ where $\id{p}\mid c$ and using the fact that $5$ is
not ramified in~$K$, we find that $c_4\equiv-5b^4\pmod{\id{p}^2}$, so
  $\id{p}^2\nmid c_4$.  Hence $E_{a,b}$ is also minimal at~$\id{p}$.
  Examples show that the reduction at such a prime may be either good
  or additive.

In the factorization $\Disc(E_{a,b}) = a^5b^5(a^2-11ab-b^2)$, the
three factors are pairwise coprime. Then for $\Disc(E_{a,b})$ to be a
prime power, two of $a$, $b$ and $a^2-11ab-b^2$ are units. Since we may
scale $a$ and~$b$ simultaneously by a unit we may assume that either
$a=1$ or $b=1$.  When $a=1$, $b=\pm1$ leads to discriminant~$-11$
while $a^2-11ab-b^2=\pm1$ leads to discriminant~$-11^5$, giving the
curves \lmfdbec{11}{a}{2} or \lmfdbec{11}{a}{3} over any field in
which $11$ is not split.  Over $\Q(\sqrt{-1})$, additionally,
$(a,b)=(1,\pm\sqrt{-1})$ gives the curves with
conductor~$(2\pm\sqrt{-1})^2$ as stated, while over $\Q(\sqrt{-3})$
none of the additional units gives a solution.  Lastly if $a$ is not a
unit we may assume that $b=1$ and require $u=a^2-11a-1\in\Om_K^*$ so
that $125+4u$ is a square in~$K$; the only possibility is $u=-1$ and
$a=11$ giving discriminant $-11^5$ again.
\end{proof}

\begin{thm}
  Let $K$ be an imaginary quadratic field of class number $1$ and $E/K$ be an elliptic
  curve of odd prime power conductor with a point of order $7$. Then
  $K=\Q(\sqrt{-3})$ and $E$ is isomorphic to
  \begin{eqnarray}
    \label{eq:curves7torsion}
    y^2 + a y = x^3 + (a-2)x^2+(1-a)x\qquad&&\text{with label \lmfdbecnf{2.0.3.1}{49.3}{CMa}{1}},
  \end{eqnarray}
where $a = \frac{1+\sqrt{-3}}{2}$.
\label{thm:7-torsion}
\end{thm}

\begin{proof}
In this case, a general elliptic curve with a $7$-torsion point is
given by
\[
E_{a,b}: y^2+(b^2+ab-a^2) xy-(a^3b^3-a^2b^4)y = x^3-(a^3b-a^2b^2) x^2 ,
\]
where $a,b \in \mathcal{O}_{K}$ and $\gcd(a,b)=1$ (see
\cite{Kubert}). Its discriminant is given by
$\Disc(E_{a,b}) = a^7 b^7 (a-b)^7 (a^3-8a^2b+5a b^2+b^3)$.  As in the
previous theorem, we find that $\gcd(\Disc(E_{a,b}), c_4(E_{a,b}))$
is not divisible by any prime except possibly those dividing~$7$, and that
the model $E_{a,b}$ is minimal even at such primes.

Since $(a,b)=1$, two of $a$, $b$ and $c=(a-b)^7 (a^3-8a^2b+5a b^2+b^3)$
are units.   If $a$ is a unit but not $b$ then we can scale so $a=1$
and now $1-b$ is a unit.  None of the possibilities gives an odd prime
power discriminant.  Similarly if $b=1$ and $a$ is not a unit. Lastly
if $a=1$ and $b$ is a unit, the only possibility which works is when
$b$ is a 6th root of unity, giving the curve as stated in the theorem
(which is isomorphic to its Galois conjugate, while not being a
base-change from $\Q$).
\end{proof}

\begin{remark}
  Note that in all the exceptional cases, the discriminant valuation is
  at most $5$, except for the curve \lmfdbec{11}{a}{2} over $\Q(\sqrt{-11})$,
  where it is $10$ (as $11$ ramifies in the extension).
\end{remark}

\section{Prime power conductor curves with rational $2$-torsion}
\label{section 2 torsion}
Our goal in this section is to classify all elliptic curves~$E$
defined over~$K$ with odd prime power conductor which have a
$K$-rational point of order~$2$.  As a result we will see
(Corollary~\ref{coro:2torsion-odddiscriminant} below) that each
isogeny class of such curves contains a curve whose discriminant has
odd valuation.  Our results here extend those of Shumbusho, who in his
2004 thesis~\cite{Shumbusho} considered elliptic curves over the same
fields as we do, with prime conductor and rational $2$-torsion.

In this section we will make essential use of the local criterion
of Kraus from \cite[Th\'eor\`eme 2]{Kraus-c4c6}, which we state here
for the reader's convenience.

\begin{prop}[(Kraus)]
   Let $K$ be a finite extension of $\Q_2$ with valuation
   ring~$\Om_K$, normalized valuation $v$ and ramification
   degree~$e=v(2)$.  Let $c_4$, $c_6$, $\Delta$ in $\Om_K$ satisfy
   $c_4^3-c_6^2=1728\Delta\not=0$.  Then there exists an integral
   Weierstrass model of an elliptic curve over~$K$ with invariants
   $c_4$ and~$c_6$ if and only if one of the following holds:
   \begin{enumerate}
     \item $v(c_4)=0$, and there exists $a_1\in\Om_K$ such that
       $a_1^2\equiv -c_6\pmod{4}$.
     \item $0<v(c_4)<4e$, and there exist $a_1,a_3\in\Om_K$ such that
       \begin{eqnarray*}
       d=-a_1^6+3a_1^2c_4+2c_6 &\equiv& 0 \pmod{16};\\
       a_3^2 &\equiv& d/16 \pmod{4};\\
       4a_1^2d &\equiv& (a_1^4-c_4)^2 \pmod{256}.
       \end{eqnarray*}
     \item $v(c_4)\ge 4e$, and there exists $a_3\in\Om_K$ such that
       $a_3^2\equiv c_6/8\pmod{4}$.
   \end{enumerate}
\label{prop:Kraus}
 \end{prop}

\subsection{Preliminaries on curves with odd conductor and rational $2$-torsion}
In this subsection and the next we
assume that $E$ is an elliptic curve defined over~$K$, with odd
conductor, and with a rational $2$-torsion point; later we will
specialize to the case where the conductor is an odd prime power.
Such an elliptic curve $E$ has an equation of the form
\begin{equation}
  \label{eq:equation}
  E_{a,b}:y^2 = x(x^2 + ax+b),
\end{equation}
where $a,b\in\Om_K$ and the given $2$-torsion point is
$(0,0)$. However, while it is easy to see that this model can be taken
to be minimal at all odd primes, we need to be more precise concerning
the primes dividing~$2$ where such a model cannot have good reduction.
To this end, we need to consider carefully the transformation from a
global minimal model for $E$ (which exists since $K$ has class
number~$1$) to this form. Let
\begin{equation}\label{eqn:min-model}
y^2+a_1xy+a_3y=x^3+a_2x^2+a_4x+a_6
\end{equation}
be a global minimal model for~$E$; its discriminant $\Disc_{\min}(E)$
is odd since $E$ has odd conductor.  To transform this model into
$E_{a,b}$ we first complete the square, then scale to make the equation
integral, and finally translate the $x$-coordinate so that the
$2$-torsion point has $x=0$.

After completing the square the right-hand side of the equation is
\begin{equation}
  \label{eq:2torsionpoly}
x^3 + (a_2+(a_1/2)^2)x^2 + (a_4+(a_1/2)a_3)x + (a_6+(a_3/2)^2)
\end{equation}
which is still integral, and minimal, at all odd primes.  Let
$\q=(\tau)$ be a prime of~$K$ dividing~$2$.  Then $a_1$ and $a_3$
are not both divisible by~$\q$, as otherwise $\q$ would divide the
discriminant, so (\ref{eq:2torsionpoly}) is not integral
at~$\q$.  To make the equation integral we scale $x$ by
$\tau^{2r}$ for some $r\ge1$ chosen to be minimal.  The minimal~$r$
depends on whether $E$ is ordinary or supersingular at~$\q$,
or equivalently whether $\vq(a_1)=0$ or $\vq(a_1)>0$.
\begin{prop}\label{Proposition:scaling}
Let $E$ be an elliptic curve of odd conductor over~$K$ with a
$K$-rational point of order~$2$ with minimal
equation~(\ref{eqn:min-model}).  Let $\q=(\tau)$ be a prime of~$K$ dividing~$2$.
\begin{enumerate}
\item If $E$ has ordinary reduction at~$\q$ (that is, if
  $\vq(a_1)=0$), then the minimal scaling to make
  (\ref{eq:2torsionpoly}) integral is $x\mapsto 2^2x$ (with scaling
  valuation $r=2e_2$).
\item If $E$ has supersingular reduction at~$\q$ (that is, if
  $\vq(a_1)>0$), then $K$ is ramified at~$2$ and
  $\vq(a_1)=1$. The minimal scaling is $x\mapsto \tau^2x$ (with $r=2$).
\end{enumerate}
In both cases, after scaling, (\ref{eq:2torsionpoly}) reduces
to~$x^2(x+u)$ modulo~$\q$ with $u$ odd.  In the supersingular case,
there is only one $K$-rational point of order~$2$, whose
$x$-coordinate (after scaling) is odd.
\end{prop}
\begin{proof}
After scaling $x$ by $\tau^{2r}$, the coefficient of $x^{3-j}$ is
multiplied by~$\tau^{2rj}$.

First suppose that $\vq(a_1)=0$ (the ordinary case).  From the
coefficient of $x^2$ in (\ref{eq:2torsionpoly}) it is immediate that
$x\mapsto 4x$ is the minimal scaling which gives integral
coefficients. After scaling, the coefficient of $x^2$ is
$u=a_1^2+4a_2$, with $\q$-valuation~$0$, and the others are
divisible by~$8$ and $16$ respectively.

Now suppose that $\vq(a_1)>0$ (the supersingular case), which
implies $\vq(a_3)=0$. If $\vq(a_1) \ge e_2$, then all
coefficients in (\ref{eq:2torsionpoly}) are $\q$-integral except
the last one, which has $\q$-valuation $-2e_2$. But then all roots
have valuation $-2e_2/3$, which is not an integer, contradicting the
fact that the polynomial has a root in~$K$. It follows that this
supersingular case can only occur if $e_2=2$ and $\vq(a_1) =1$.
The coefficients in~(\ref{eq:2torsionpoly}) now have valuations~$-2$,
$-1$, $-4$, from which it follows that the roots (whether in~$K$ or an
extension) have valuations~$-2$, $-1$, $-1$; since the $x$-coordinates
of non-integral $K$-rational points must have even valuation, there
can be only one $K$-rational point of order~$2$, with $x$-coordinate
of valuation~$-2$.  To achieve integrality we must scale~$x$
by~$\tau^2$, after which the cubic reduces to $x^2(x+1)$
modulo~$\q$ and the $x$-coordinate of the $K$-rational point of
order~$2$ is odd.
\end{proof}

We will refer to the two cases of this proposition as ``the ordinary
case'' and ``the supersingular case'' respectively.

The model $E_{a,b}$ has invariants
\begin{equation}\label{eqn:Eab-invariants}
  \Delta=2^4b^2(a^2-4b), \qquad
  c_4 = 2^4(a^2-3b), \qquad
  c_6 = 2^5a(9b-2a^2).
\end{equation}
We set $\Disc(E)=b^2(a^2-4b)$, and compare this with
$\Disc_{\min}(E)$.  In the ordinary case, $\Delta=2^{12}\Disc_{\min}(E)$
so $\Disc(E)=2^8\Disc_{\min}(E)$; in the supersingular case,
$\Delta=\tau^{12}\Disc_{\min}(E)$ and $\Disc(E)=\pm2^2\Disc_{\min}(E)$,
with sign $-1$ for $\Q(\sqrt{-1})$ when $\tau=1+\sqrt{-1}$ and $+1$
for $\Q(\sqrt{-2})$ when $\tau=\sqrt{-2}$.

\begin{coro}\label{Cor:ab-q-valuations}
Let $(x_0,0)$ be the coordinates of the given $2$-torsion point on
the scaled model, so $x_0\in\Om_K$.  We obtain a model of the
form~$E_{a,b}$ by shifting the $x$-coordinate by~$x_0$.
\begin{enumerate}
\item In the ordinary case, if $\vq(x_0)>0$ then we obtain
  $(a,b)$ with $(\vq(a), \vq(b)) = (0, 4e_2)$, while if
  $\vq(x_0)=0$ then $(\vq(a), \vq(b)) = (e_2, 0)$.
\item In the supersingular case we always have $\vq(x_0)=0$,
  and $(\vq(a),\vq(b))=(k,0)$ with $k\ge3$.  (We include
  the possibility that $a=0$ here.)
\end{enumerate}
\end{coro}
\begin{proof}
After the shift by~$x_0$ we have $\q\mid b$ if $x_0$ reduces to
the double root modulo~$\q$ and $\q\mid a$ otherwise; $\q$
does not divide both since there is no triple root modulo~$\q$.
In the supersingular case, $\q$ must divide~$a$.

Suppose that $\q\mid b$.  Then we are in the ordinary case, and
$8e_2=\vq(\Disc(E))=2\vq(b)$ so $\vq(b)=4e_2$.

Alternatively, suppose that $\q\mid a$.  Then in the ordinary
case, $\vq(a^2-4b)=\vq(\Disc(E))=8e_2>2e_2=
\vq(4b)$, so $\vq(a^2)=2e_2$ and $\vq(a)=e_2$.
In the supersingular case, $\vq(a^2-4b)=4=\vq(4b)$ so
$\vq(a)\ge2$; however it is easy to see that $\vq(a)=2$
leads to a contradiction since the residue field at~$\q$ has only
size~$2$.
\end{proof}

Note that $a=0$ can only happen in the supersingular case.  Such
curves have CM by $\Z[\sqrt{-1}]$ and were considered in
Section~\ref{section CM}.

In what follows, it would be enough to determine curves up to quadratic
twist, since given one elliptic curve it is straightforward (see
\cite{CremonaLingham}) to find all of its twists with good reduction
outside a fixed set of primes.  The quadratic twists of~$E_{a,b}$ have
the form $E_{\lambda a,\lambda^2b}$ for $\lambda\in K^*$.  Taking
$\lambda=\mu^{-1}$ with $\mu\in\Om_K$, where $\mu\mid a$ and
$\mu^2\mid b$, we obtain a curve with a smaller discriminant, by a
factor $\mu^6$.  In our situation of curves with odd conductor, such a
factor~$\mu$ must be odd, supported on primes of bad reduction, and
also square-free (by minimality of the equation at all odd primes).
However, if $\p=(\pi)$ is an odd prime factor of the conductor
such that $\pi\mid a$ and $\pi^2\mid b$, it can happen\footnote{This
  can be avoided over~$\Q$, since for every odd prime ideal $(p)$,
  either $\pm p\equiv1\pmod{4}$, and $\Q(\sqrt{\pm p})$ is
  unramified at~$2$.} that the twist
$E^\pi$ acquires bad reduction at a prime~$\q$ dividing~$2$: this
is the case if $\q$ ramifies in $K(\sqrt{\pi})$.

For a prime~$\p=(\pi)$ we say that the pair $(a,b)$ is
\emph{minimal} at~$\p$ (or~$\pi$) if either $\pi\nmid a$ or
$\pi^2\nmid b$.  When $(a,b)$ are the parameters obtained from a curve
of odd conductor, we have already seen that $(a,b)$ is minimal
at~$\q$ for all $\q\mid2$ since either $a$ or~$b$ is not
divisible by~$\q$, while for the primes~$\p$ dividing the
conductor, $(a,b)$ may not be $\p$-minimal.  However, as already
observed, we can always assume that either $\pi^2\nmid a$ or
$\pi^4\nmid b$.

Since we cannot assume that a minimal twist of a curve with odd
conductor still has odd conductor, we will need to consider curves
with non-minimal $(a,b)$.  In the prime power conductor case, this
means that we will consider curves in sets of four twists, a base
curve $E=E_{a,b}$ which is $\p$-minimal, may have bad reduction at
primes dividing~$2$, and may even have good reduction at~$\p$; and
the twists $E^s$ of $E$ by~$s\in\{\varepsilon, \pi, \varepsilon\pi\}$,
where $\p=(\pi)$.

For example, over $K=\Q(\sqrt{-2})$ we have seen in a previous
section that the elliptic curve \lmfdbec{256}{d}{2}, which has conductor
$(\sqrt{-2})^{10}$, has infinitely many quadratic twists of odd prime
square conductor.

\subsection{Classification of curves with odd conductor and $2$-torsion}

We continue with the notation of the previous subsection: $E$ is an
elliptic curve with odd conductor and a $K$-rational point of
order~$2$, and $(a,b)$ are parameters for the model $E_{a,b}$ for $E$
constructed above.  Write $(2)=\q_a\q_b$ where $\q_a$,
with generator $\tau_a$, (respectively $\q_b$, with generator
$\tau_b$), is divisible only by primes dividing $a$ (respectively
$b$), and $2=\tau_a\tau_b$.  For $K\not=\Q(\sqrt{-7})$ we have
$(\tau_a,\tau_b)=(1,2)$ or~$(2,1)$ according to whether
$(\vq(a), \vq(b)) = (0, 4e_2)$ or~$(e_2, 0)$ for the
unique prime~$\q$ dividing~$2$ in the ordinary case, and also
$(\tau_a,\tau_b)=(2,1)$ in the supersingular case.

The following result completely classifies curves with odd conductor
and a point of order~$2$ in terms of the solutions to a certain
equation (\ref{eq:twotorsioncases}).

\begin{thm}
  Let $E$ be an elliptic curve defined over~$K$, with a $K$-rational
  $2$-torsion point and odd conductor. Let $D=\Disc_{min}(E)$.  Set
  \begin{enumerate}
    \item $P=\gcd(D,b,a^2-4b) = As^2$ with $A$ square-free;
    \item $B=(a^2-4b)/(\tau_a^2P)$;
    \item $C=4b/(\tau_a^2P)$;
    \item $\tilde{a}=as/(\tau_aP)$.
  \end{enumerate}
  Then
  \begin{equation}
  \tilde{a}^2A=B+C
   \label{eq:twotorsioncases}
  \end{equation}
 and $\tilde{a},A,B,C\in\Om_K$ satisfy the following conditions:
 \begin{enumerate}
 \item $\gcd(B,C)=1$ and $A$ is square-free;
 \item $A,B,C$ are only divisible by primes dividing $2D$.
 \end{enumerate}
Furthermore,
 \begin{enumerate}
 \item For each prime $\p \mid D$ with $k=\vp(D)$ and $k'=k-6$:
\[
  (\vp(A),\vp(B),\vp(C)) = (0,k,0), (0,k',0),
(0,0,\frac12k), (0,0,\frac12k')  \quad\text{or}\quad (1,0,0);
\]
\item In the ordinary case, for each prime $\q \mid 2$,
\[
  (\vq(A),\vq(B),\vq(C)) = (0,6e_2,0)
\quad\text{or}\quad (0,0,6e_2).
\]
\item In the supersingular case, for $\q \mid 2$,
$(\vq(A),\vq(B),\vq(C)) = (0,0,0)$.
 \end{enumerate}

 Conversely, given integral $A,B,C,\tilde{a}$ satisfying
 $\tilde{a}^2A=B+C$ and the above conditions, if we set
 $a=2A\tilde{a}/\gcd(2,C)$ and $b=AC/\gcd(2,C)^2$ then $E_{a,b}$ and
 its twists $E_{sa,s^2b}$ for all square-free $s\in\Om_K$
 dividing~$D$, all have good reduction outside $2D$.
\label{thm:classification2torsion}
\end{thm}
\begin{proof}
We have $a^2=(a^2-4b)+4b=B_0+C_0$ where $B_0=a^2-4b=\Disc(E)/b^2$ and
$C_0=4b$, and consider $\gcd(B_0,C_0)$ one prime at a time, these
primes all being divisors of $2D$.

First consider odd primes~$\p=(\pi)$.  These contribute to~$P$ if
$\p\mid B_0$ and $\p\mid C_0$, so that $\p\mid b$ and
$\p\mid a$.  In general, the contribution to~$P$ from~$\p$
is~$\pi^j$ where $j=\min\{\vp(b),\vp(a^2-4b)\}$, with
$0\le j\le3$ as in Table~\ref{table:p-valuations}, where
$k=\vp(D)$.  The entries above the line correspond to $(a,b)$
being $\p$-minimal while those below are non-minimal.  We include
$a=0$ as a possibility in each row with \lq${}\ge{}$\rq\  in the first column.
\begin{table}[h!]
  \begin{tabular}{||c|c|c||c|c||c|c|c||}
\hline
$\vp(a)$ & $\vp(b)=\vp(C_0)$ &
$\vp(a^2-4b)=\vp(B_0)$ & $j$ & $k$& $\vp(A)$ &
$\vp(B)$ & $\vp(C)$  \\
\hline
$0$ & $0$ & $\ge0$ & $0$ & $\ge0$ & $0$ & $k$ & $0$ \\
$\ge1$ & $0$ & $0$ & $0$ & $0$ & $0$ & $0$ & $0$ \\
$0$ & $\ge1$ & $0$ & $0$ & $\ge2$, even & $0$ & $0$ & $\frac12k$ \\
$\ge1$ & $1$ & $1$ & $1$ & $3$ & $1$ & $0$ & $0$ \\
\hline
$1$ & $2$ & $\ge2$ & $2$ & $\ge6$ & $0$ & $k-6$ & $0$ \\
$\ge2$ & $2$ & $2$ & $2$ & $6$ & $0$ & $0$ & $0$ \\
$1$ & $\ge3$ & $2$ & $2$ & $\ge8$, even & $0$ & $0$ & $\frac12(k-6)$ \\
$\ge2$ & $3$ & $3$ & $3$ & $9$ & $1$ & $0$ & $0$ \\
\hline
\end{tabular}
\caption{Possible parameter valuations at an odd prime}
\label{table:p-valuations}
\end{table}
Let $P$ be the product of~$\pi^j$ over all these odd prime ideals.  We
can write $P=As^2$ with $A$ and~$s$ square-free (since $j\le3$ in all
cases); $P$ is the odd part of $\gcd(a^2-4b,4b)$, which is
$\gcd(D,b,a^2-4b)$.

Now consider the prime (or primes) $\q\mid 2$, which divide $a$
or~$b$ but not both by Corollary~\ref{Cor:ab-q-valuations}.  The
possible valuations are given in table~\ref{table:q-valuations}.
Since $e_2=2$ in the supersingular case, this prime(s) contributes
$\tau_a^2$ to $\gcd(B_0,C_0)$.
\begin{table}[h!]
    \scalebox{0.9}{
\begin{tabular}{||c||c|c||c|c||c|c|c||}
\hline
Case & $\vp(a)$ & $\vp(b)$ & $\vp(B_0)=\vp(a^2-4b)$ &
       $\vp(C_0)=\vp(4b)$ & $\vp(A)$ & $\vp(B)$ &
       $\vp(C)$ \\
\hline
\multirow{2}{6em}{Ordinary} & $0$ & $4e_2$ & $0$ & $6e_2$ & $0$ & $0$ & $6e_2$ \\
& $e_2$ & $0$ & $8e_2$ & $2e_2$ & $0$ & $6e_2$ & $0$ \\
\hline
Supersingular & $\ge3$ & $0$ & $4$ & $4$ & $0$ & $0$ & $0$ \\
\hline
\end{tabular}}
\caption{Possible parameter valuations at a prime dividing 2}
\label{table:q-valuations}
\end{table}

Hence $\gcd(B_0,C_0)= \tau_a^2P$; dividing through gives
$\tilde{a}^2A=B+C$ with $B,C,\tilde{a}$ as given.  In the
factorization $P=As^2$, $A$ is the product of those odd $\pi$ for
which $j$ is odd while $s$ is the product of those for which $j\ge2$;
both are square-free divisors of~$D$.

For all primes~$\p$ dividing $2D$, the values of $\vp(A)$,
$\vp(B)$, $\vp(C)$ in the tables may easily be deduced
from the previous columns.

For the converse, it suffices to observe that with $a,b$ as defined we
have $b^2(a^2-4b)=4A^3BC^2/\gcd(2,C)^6$, whose support lies in $2D$,
and hence determines a base curve $E_{a,b}$, which has good reduction
away from~$2D$.  The same is true of twists by square-free~$s$
dividing~$D$.
\end{proof}

\begin{remark}  We can scale solutions $(A,B,C)$ to
  (\ref{eq:twotorsioncases}) by units without affecting the
  conditions, and scaling by squares of units gives isomorphic curves.
  Since $K$ has finitely many units, for each odd~$D\in\Om_K$, we can
  use Theorem~\ref{thm:classification2torsion} to compute all curves
  of discriminant $D$ with a $K$-rational two-torsion point.  This
  will be our strategy in the following subsections, where we restrict
  to the case where there is only one odd prime factor.
\end{remark}

Recall that each curve $E_{a,b}$ has a $2$-isogenous curve
$E_{-2a,a^2-4b}$, via the $2$-isogeny with kernel $(0,0)$, which has
the same conductor as~$E$.  At odd primes~$\p$ it is immediate
that $(a,b)$ is minimal at~$\p$ if and only if
$(a',b')=(-2a,a^2-4b)$ is minimal.  When $E_{a,b}$ has odd conductor
with $(a,b)$ given by our construction, at primes $\q$
dividing~$2$ we see that in the ordinary case, $(\vq(a),
\vq(b)) = (0, 4e_2)$ implies $(\vq(-2a),
\vq(a^2-4b)) = (e_2, 0)$, while conversely if $(\vq(a),
\vq(b)) = (e_2, 0)$ then $(\vq(a'), \vq(b')) =
(2e_2, 8e_2)$ and the associated minimal pair is $(a'/\tau^{2e_2},
b'/\tau^{4e_2})$ with valuations $(0,4e_2)$.  In the supersingular
case with valuations $(\ge3,0)$, a minimal pair for the isogenous
curve is $(a'/(-2), b'/4) = (a,(a/2)^2-b)$ with the same valuations
as~$(a,b)$.

\begin{prop}
  Following the notation of Theorem~\ref{thm:classification2torsion},
  if $E_{a,b}$ is a curve with odd conductor associated to a triple
  $(A,B,C)$ satisfying~(\ref{eq:twotorsioncases}), then the
  $2$-isogenous curve $E_{-2a,a^2-4b}$ has associated triple
  $(A,C,B)$.
\end{prop}
\begin{proof}
  It is clear that the parameter $A$ is the same for both curves, and
  then a straightforward computation gives the result, using the
  remarks about minimality stated above.
\end{proof}

Given an odd prime ideal $\p=(\pi)$ of bad reduction, we have two
different situations depending on whether $\vp(A) =0$ or
$\vp(A)=1$. In the first case, one curve in each isogenous pair has
double the discriminant valuation of the other, and one of the two
isogenous pairs, which are $\pi$-twists of each other, has good or
multiplicative reduction at~$\p$ (when the parameter $s$ of
Theorem~\ref{thm:classification2torsion} is not divisible by~$\pi$),
while the other has additive reduction (when~ $s$ is divisible
by~$\pi$).  In the second case, all curves have additive reduction
at~$\p$, one isogenous pair (with $\pi\nmid s$) has discriminant
valuation~$3$ and the other (with $\pi\mid s$) has valuation~$9$.

Each solution to~(\ref{eq:twotorsioncases})  falls into one of
these cases, according to whether $B+C$ has even or odd valuation
at~$\p$, unless $B+C=0$, corresponding to $a=0$, which we have treated
separately.  When seeking curves with conductor a power of the odd
prime~$\p$ in subsequent subsections, we will also treat separately
solutions to~(\ref{eq:twotorsioncases}) where all of $A$, $B$ and $C$
are units at~$\p$, since in such cases we cannot recover~$\p$ from the
solution.  Such cases occur when an elliptic curve~$E$ has
conductor~$\p^2$, so has additive reduction at~$\p$, but is a
quadratic twist of a curve~$E_0$ with good reduction at~$\p$.  Here we
must consider twists of~$E_0$ by \emph{all} odd primes to see which
have good reduction above~$2$.

In the next three subsections we will determine all elliptic curves
with odd prime power conductor~$\p^r$, treating first this ``good
twist'' case where $E$ has a twist with good reduction at~$\p$ (this
includes the case $a=0$), then the ``additive twist'' cases where all
twists have additive reduction at~$\p$ (here $A=\pi$) and lastly the
``multiplicative twist'' case where $E$ has a twist with
multiplicative reduction at~$\p$ (here $A=1$).

In each case we determine which $(B,C)$ pairs give an appropriate
solution to (\ref{eq:twotorsioncases}), thus obtaining a ``base
curve'' $E_{a,b}$, and then determine which of its twists have good
reduction at primes dividing~$2$.  
\subsection{Curves with odd prime power conductor: the good twist case}
In this subsection we determine all elliptic curves defined over one
of the fields~$K$ with $K$-rational $2$-torsion and conductor a power
of an odd prime~$\p$, such that a quadratic twist of~$E$ by a
generator~$\pi$ of~$\p$ has good reduction at~$\p$.  In fact all such
curves have CM by an order in~$K$ and have been fully described previously.

\begin{thm}\label{thm:goodtwist}
  Let $K$ be an imaginary quadratic field with class number $1$, and
  let $\p$ be an odd prime of~$K$. Let $E$ be an elliptic curve
  defined over~$K$, with a $K$-rational point of order~$2$ of
  conductor a power of~$\p$.  If $E$ has a quadratic twist with good
  reduction at~$\p$ then $E$ belongs to one of the following complex
  multiplication families as studied in Section~\ref{section CM}:
  \begin{enumerate}
  \item $K=\Q(\sqrt{-7})$ and~$E$ is a twist of the base-change to~$K$
    of one of the curves in the isogeny class \lmfdbeciso{49}{a} over $\Q$,
    with conductor $\p^2=(\pi)^2$ where $\pi\equiv 1\pmod4$; or
  \item $K=\Q(\sqrt{-1})$ and~$E$ is a twist of the base-change to~$K$
    of one of the curves in the isogeny class \lmfdbeciso{64}{a} over $\Q$.
    $E$ has equation $y^2=x(x^2+b)$ and conductor~$\p^2=(\pi)^2$,
    where $b\equiv-1\pm2i\pmod8$ with~$b=\pi$ or $b=\pi^3$; or
  \item $K=\Q(\sqrt{-2})$ and $E$ is a twist of the base-change to $K$
    of one of the curves in the isogeny class \lmfdbeciso{256}{a} over $\Q$,
    with conductor $\p^2=(\pi)^2$ where $\pi \equiv \pm 1 +\sqrt{-2}
    \pmod 4$, for $(a,b)=(2\pi\sqrt{-2},-\pi^2)$.
  \end{enumerate}
In each case, the elliptic curves have CM by an order in the field of
definition~$K$: their $j$-invariants are either $-15^3$ or $255^3$ in
the first case, either $12^3$ or $66^3$ in the second, and in the
third they have $j=20^3$ and CM by $\Z[\sqrt{-2}]$.
\end{thm}

\begin{remark}
There are four curves in the isogeny class \lmfdbeciso{49}{a} over $\Q$, linked
by $2$-{} and $7$-isogenies and in two pairs of $-7$-twists, so that
over $\Q(\sqrt{-7})$ they become isomorphic in pairs.  The first two
are \lmfdbec{49}{a}{1}, which has parameters $(a,b)=(21,112)$, and \lmfdbec{49}{a}{2}, with
$(a,b)=(-42,-7)$; the other two are their $-7$-twists and are
$7$-isogenous to these.
\end{remark}

\begin{proof}
  By Theorems \ref{theorem:CMnorootsofunity} and
  \ref{theorem:CMGaussian} we know that the three families of curves
  do have good reduction away from $\id{p}$, hence we are led to prove
  that these are the unique ones.

  In the notation of Table~\ref{table:p-valuations} such curves have
  discriminant valuation~$k=6$ so come from solutions with
  $\p$-valuations given in the first two lines of the second half of
  the table.  Hence $B$ and $C$ are not divisible by~$\pi$, and from
  Table~\ref{table:q-valuations} they have valuation~$6e_2$ or~$0$ at
  the prime(s) above~$2$ in the ordinary case, while in the
  supersingular case they are units.  Up to scaling by units, and
  interchanging~$B$ and~$C$ (corresponding to applying a $2$-isogeny),
  we reduce to considering the following finite number of
  possibilities for~$(B,C)$:
  \begin{enumerate}
    \item over all fields, $(B,C) = (64\eta,1)$ with~$\eta\in\Om_K^*$;
    \item over $\Q(\sqrt{-7})$ where $2=\tau\overline{\tau}$ with
      $\tau=\frac{1+\sqrt{-7}}{2}$, $(B,C) =
      (\pm\tau^6,\overline{\tau}^6)$;
    \item over $\Q(\sqrt{-1})$ and $\Q(\sqrt{-2})$, $(B,C)=(\pm1,1)$
      (the supersingular case).
  \end{enumerate}
  For a solution we require $B+C$ to be either~$0$ or a non-zero
  square times a unit.

  Case (1) yields no solutions with $\eta=1$ since $B+C=65=5\cdot13$
  is not a square since neither $5$ nor~$13$ ramifies in any of the
  fields.  Taking $\eta=-1$ in (1) gives $B+C=-63=3^2\cdot7$, which is
  valid when $7$ is ramified, and leads to the base curves~$E=E_{a,b}$
  with $(a,b)=(6\sqrt{-7},1)$ (and its Galois conjugate). Such curves
  lie in the first family.

  A simple check shows that none of the additional units in
  $\Q(\sqrt{-1})$ or $\Q(\sqrt{-3})$ gives a value of $B+C$ of the
  required form.

  In case (2) we have $B+C = \pm\tau^6\pm\overline{\tau}^6 \in
  \{\pm9,\pm5\sqrt{-7}\}$, giving a potential solution with $A=\pm6$
  and~$b=\overline{\tau}^6$ (or its Galois conjugate).  Taking
  $(a,b)=(6,\tau^6)$ we find a twist of~\lmfdbec{49}{a}{1} and hence no curves
  not already encountered.

  In case (3) with $B+C=0$ we obtain curves with $a=0$.  All such
  curves have CM by $\Z[\sqrt{-1}]$, hence we get curves in the second
  family from Theorem~\ref{theorem:CMGaussian}.

  In case (3) with $B+C=2$ we obtain a solution when $2$ is ramified,
  with base curve~$E=E_{a,b}$ where $(a,b)=(2(1+i),i)$
  or~$(2\sqrt{-2},-1)$. Both cases are isomorphic to the curve \lmfdbec{256}{a}{1}
  with CM by $\ZZ[\sqrt{-2}]$. Then we get the third family from
  Theorem~\ref{theorem:CMEisenstein} and
  Corollary~\ref{coro:differentfields}.
\end{proof}

\subsection{Curves with odd prime power conductor: the additive twist case}

We continue to consider elliptic curves~$E$ whose conductor is a power
of the odd prime $\p=(\pi)$, using
Theorem~\ref{thm:classification2torsion} to find all such curves by
considering solutions to the parametrizing
equation~(\ref{eq:twotorsioncases}).

In this subsection, we consider the ``additive twist'' case in which
the parameter $A$ is divisible by~$\pi$ so that the curves and their
twists by~$\pi$ and by units all have additive reduction at~$\p$.  The
discriminant valuations are~$3$ or~$9$.  We find that the only such
curves are again the base changes of CM elliptic curves over $\Q$ with
conductor~$49$, but unlike the previous subsection, $K$ must be one of
the six fields in which $7$ is inert.  This corresponds to looking at
elliptic curves with CM by an order in $K$ over a field $L \neq K$, which
furthermore have a $2$-torsion point and odd prime power conductor. By
the results of Section~\ref{section CM} (specifically
Corollary~\ref{coro:differentfields}), we have to restrict to odd values
of $d$, and the unique curve with a $2$-torsion point corresponds to
$d=7$.

\begin{thm} Let $K$ be an imaginary quadratic field with class number
  $1$, and let $\p$ be an odd prime of~$K$. Let $E$ be an elliptic
  curve defined over~$K$, with a $K$-rational point of order~$2$ and
  conductor  a power of~$\p$, such that no quadratic twist of~$E$
  has good or multiplicative reduction at~$\p$.  Then
  \begin{enumerate}
  \item $K=\Q(\sqrt{-d})$ for $d=1,2,11,43,67,163$, $E$ has
    conductor~$\p^2$ where $\p=(7)$, and~$E$ is a base-change to~$K$
    of one of the curves in the isogeny class \lmfdbeciso{49}{a} over $\Q$.
  \end{enumerate}
\label{thm:additivetwist}
\end{thm}

\begin{proof}
  Inspecting Table~\ref{table:p-valuations}, we are led to the same
  set of pairs $(B,C)$ as considered in the proof of
  Theorem~\ref{thm:goodtwist}, except that now we require $B+C$ to be
  nonzero, with odd valuation at exactly one odd prime~$\p$.  We use
  the same numbering of cases as before and recall that these are all
  possibilities, up to scaling by units and switching $B$ and~$C$.

  Case (1), where $B+C\in\{\pm63,\pm65\}$ again yields no solutions
  with $B+C=65=5\cdot13$ since neither $5$ nor~$13$ ramifies in any of
  the fields.  However, $B+C=-63=3^2\cdot7$ is valid when $7$ is inert
  in~$K$.  This gives the base curve with $(a,b)=(-42,-7)$, which is
  the elliptic curve defined over~$\Q$ with label \lmfdbec{49}{a}{2}.  Note
  that this curve also appeared in the good twist case
  over~$\Q(\sqrt{-7})$, but here we require $7$ to be inert.  The
  quadratic twist by~$-7$ (with label~\lmfdbec{49}{a}{4}) also has good
  reduction away from~$7$, so all four curves in the isogeny
  class~\lmfdbeciso{49}{a} have conductor~$(7)^2$ over the fields listed.

  A simple check shows that none of the additional units in
  $\Q(\sqrt{-1})$ or $\Q(\sqrt{-3})$ gives a value of $B+C$ of the
  required form.

  In case (2) we have $B+C = \pm\tau^6\pm\overline{\tau}^6 \in
  \{\pm9,\pm5\sqrt{-7}\}$.  Since $5$ is inert this gives no
  solutions.

  Case (3), with $B+C\in\{0,\pm2\}$, also provides no solutions.
\end{proof}

\subsection{Curves with odd prime power conductor: the multiplicative
  twist case} We now consider curves of odd prime power
conductor~$\p^r$ which in our parametrization have $A=1$, such that
the base curve $E_{a,b}$ (with $\p$-minimal $(a,b)$) has
multiplicative reduction at~$\p$.

The main result of this subsection is that these elliptic curves are
of two types, up to quadratic twist by a generator of~$\p$:
\begin{enumerate}
\item one of a finite number of ``sporadic'' curves, with conductor
  either a prime dividing~$17$ (over all fields where~$17$ does not
  split), or a prime of norm~$257$ over $\Q(\sqrt{-1})$ only, or a
  prime of norm~$241$ over $\Q(\sqrt{-3})$ only;
\item one of a family analogous to the Setzer-Neumann family
  over~$\Q$.
\end{enumerate}

The sporadic curves are all given by a more general construction which
we discuss first.  Over any number field~$K$ let $u\in
K\setminus\{0,-16\}$ and define $E_u = E_{-(u+32)/4,u+16}$, an
elliptic curve with invariants $c_4=u^2+16u+256$,
$c_6=(u-16)(u+8)(u+32)$ and discriminant $\Delta_u=u^2(u+16)^2$.

\begin{lemma}
$E_u$ has full $2$-torsion over~$K$; the three curves $2$-isogenous
  to~$E$ are isomorphic to $E_{a,b}$ for
  $(a,b)=(2(u-16),u^2+32u+256)$, $(2(u+32),u^2)$ and $(u+8,16)$, with
  discriminants $-u(u+16)^4$, $u^4(u+16)$, and $u(u+16)$ respectively.
\end{lemma}
\begin{proof}
Elementary: note that $\Delta_u$ is a square.
\end{proof}

In fact the family of curves~$E_u$ is the universal family of elliptic
curves with full $2$-torsion over~$K$, as it is easy to check that
$E_u$ has Legendre parameter $\lambda=(u+16)/u$.  Our reason for
writing the family this way is that if we specialize the parameter $u$
to a unit with certain properties, then we obtain elliptic curves with
square-free odd conductor.

\begin{prop}\label{prop:sporadic-family}
Let $K$ be any number field and $u\in\Om_K^*$.  The quadratic twist
$E_u^{(-u)}$ of $E_u$ by $-u$, together with its three $2$-isogenous
curves, is semistable with bad reduction only at primes dividing
$u+16$.  The same is true of $E_u$ itself if $-u$ is congruent to a
square modulo~$4$.
\end{prop}

\begin{proof}
From the invariants given above we see that $\Delta_u$ is only
divisible by primes dividing $u+16$ which is odd, and that $\Delta_u$
is coprime to $c_4$ so the reduction is multiplicative at all bad
primes.  Also since $c_4$ and $c_6$ are odd the condition that $c_4$
and $c_6$ are the invariants of an integral model, which then has good
reduction at primes dividing~$2$, is that $-c_6$ is a square
modulo~$4$, which is the case when $-u$ is a square modulo~$4$ since
$c_6\equiv u^3\pmod4$.  Twisting by $-u$ gives a curve whose
$c_6\equiv-u^6\pmod4$ which satisfies Kraus's condition
unconditionally.
\end{proof}

For example, over $\Q$ we take $u=1$ and find that $E_{1}^{(-1)}$ is
the elliptic curve \lmfdbec{17}{a}{2} of conductor~$17$, with $2$-isogenous curves
\lmfdbec{17}{a}{1}, \lmfdbec{17}{a}{3} and~\lmfdbec{17}{a}{4}. Since
$17 \equiv 1 \pmod 4$, its quadratic
twists by $17$ also have good reduction at $2$. Taking $u=-1$ gives
curves of conductor $15$, which are not relevant for us.

More generally we consider the curves given by this proposition over
imaginary quadratic fields, for units $u$ such that $u+16$ is a prime
power, so that we obtain curves of prime conductor.  When $\pm1$ are
the only units, the only case is the one just considered with $u=1$,
leading to curves whose conductors are divisible only by the primes
above~$17$, which are primes except when $17$ splits in~$K$. Since
$17 \equiv 1 \pmod 4$, the quadratic twists of such curves also have
good reduction at $2$.

Over $K=\Q(\sqrt{-1})$ we can also take $u=\pm \sqrt{-1}$ since $16\pm
\sqrt{-1}$ have prime norm~$257$.  This gives~$8$ elliptic
curves,
$4$ in one isogeny class \lmfdbecnfiso{2.0.4.1}{257.1}{a} with
conductor $\p=(16+\sqrt{-1})$, linked by $2$-isogenies, and their
Galois conjugates in isogeny class \lmfdbecnfiso{2.0.4.1}{257.2}{a}.
The quadratic twists by $1 \pm 16\sqrt{-1}$ have
good reduction at $2$ and give curves of conductor $\p^2$ in isogeny classes
\lmfdbecnfiso{2.0.4.1}{66049.1}{a} and \lmfdbecnfiso{2.0.4.1}{66049.3}{a}.

Over $K=\Q(\sqrt{-3})$ let $\varepsilon$ be a $6$th root of unity
generating the unit group.  Taking $u=\varepsilon^2$ or its Galois
conjugate, we obtain elliptic curves with prime conductors~$\p$ of
norm~$241$.  Again there are two Galois conjugate isogeny
classes \lmfdbecnfiso{2.0.3.1}{241.1}{a} and \lmfdbecnfiso{2.0.3.1}{241.3}{a},
each containing $4$ elliptic curves linked by $2$-isogenies. The
quadratic twists
by $16 \pm u$ have good reduction at $2$, conductor~$\p^2$, in isogeny
classes \lmfdbecnfiso{2.0.3.1}{58081.1}{a} and \lmfdbecnfiso{2.0.3.1}{58081.3}{a}.

The next result shows that, apart from these sporadic cases, all
elliptic curves with odd prime conductor and rational $2$-torsion come
from an analogue of the Setzer-Neumann family over~$\Q$.

\begin{thm}
  Let $K$ be an imaginary quadratic field with class number~$1$, and
  $\varepsilon$ a generator of its unit group.  Let $E$ be an elliptic
  curve defined over $K$ with conductor an odd prime power~$\p^r$ and
  a $K$-rational $2$-torsion point.  Assume that $E$ has a quadratic
  twist with multiplicative reduction at $\p$.  Then $E$ is either
\begin{enumerate}
\item one of the sporadic curves listed above, where $\p$ has
  norm~$17$ (over all fields), or $257$ (over $K=\Q(\sqrt{-1})$ only) or
  $241$ (over $K=\Q(\sqrt{-3})$ only); or
\item isomorphic or $2$-isogenous to $E_{a,b}$ where
  $b=16\varepsilon$ and $a$ satisfies an equation of the form
  \[
  a^2 = u\pi^r + 64\varepsilon,
  \]
  with $r$ odd, $u$ a unit and $u \pi^r \equiv 1
  \pmod{\frac{8}{e_2}}$; or
\item the quadratic twist by $u\pi$ of the previous case, without
  any congruence condition.
\label{thm:multiplicativecase}
\end{enumerate}
\end{thm}

\begin{proof}
We start with the observation that in each case $\varepsilon$ is not
congruent to a square modulo~$4$, which may be checked easily and
which will be used repeatedly.

As before we use Theorem~\ref{thm:classification2torsion} to first
find the curves with minimal parameters, arising from solutions to
(\ref{eq:twotorsioncases}) with $A=s=1$.  Up to $2$-isogeny, we may
assume that $a=\tilde{a}$ is odd, that $B$ is odd and $C$ divisible
by~$64$ with $C/64$ odd, except in the case $K=\Q(\sqrt{-7})$ where
$B$ and $C$ are each divisible by the $6$th power of one of the two primes
dividing~$2$, or the supersingular case over $\Q(\sqrt{-1})$ or
$\Q(\sqrt{-2})$.  We will leave these last cases to the end.

Scaling by squares of units, we must solve each of the following
equations:
\begin{eqnarray}
  a^2 = P+64 \label{eqn:case1}\\
  a^2 = 1+64P \label{eqn:case2}\\
  a^2 = P+64\varepsilon \label{eqn:case3}\\
  a^2 = \varepsilon +64P \label{eqn:case4}
\end{eqnarray}
where $P$ is an odd prime power, i.e. an element of $\Om_K$ with
precisely one prime factor. We immediately see that (\ref{eqn:case4})
has no solution modulo~$4$.

(\ref{eqn:case1}) factors as $(a-8)(a+8)=P$.  Without loss of
generality (changing $a$ for $-a$ if necessary) we have $P\mid(a-8)$;
writing $a=8+Pt$ leads to $t(16+Pt)=1$, so $t$ is a unit, and
$16-t^{-1}=-Pt$.  Setting $u=-t^{-1}$ leads to one of the sporadic
cases (we have one of the curves $2$-isogenous to $E_{-t^{-1}}$) and
its quadratic twists.

(\ref{eqn:case2}) factors as $(a-1)(a+1)=64P$.  Now $P$ divides one
factor, and also one factor is divisible exactly by~$2$, the other
by~$32$.  By symmetry this gives two cases to consider: if $a=1+32Pt$
with $t$ odd then $t(1+16Pt)=1$ so $t$ is a unit and $16Pt=t^{-1}-1$
which is impossible.  Otherwise $a=1+2Pt$ with $t$ odd, and
$t(1+Pt)=16$ so again $t$ is a unit and we have a sporadic case (a
twist of $E_{-t}$).

In (\ref{eqn:case3}) we divide according to whether the valuation~$r$
of $P$ is even or odd.  If even then we must have $P=Q^2$ with $Q$ a
prime power, since $P=\varepsilon Q^2$ gives a contradiction
modulo~$4$.  Now $(a-Q)(a+Q)=64\varepsilon$; by symmetry $a=Q+32t$
with $t$ odd, so $t(Q+16t)=\varepsilon$, leading to the third sporadic
case (a twist of $E_{-\varepsilon t^2}$).

Otherwise in (\ref{eqn:case3}) we have $P=u\pi^r$ with $u$ a unit and
$r$ odd, leading to the Setzer-Neumann family. Recall that $c_4$ is
odd and $2c_6=a(9b-2a^2)$, hence Proposition~\ref{prop:Kraus} implies
that $a \equiv \square \pmod 4$ so $a^2 \equiv 1 \pmod 8$ if $2$ is
unramified in $K$ and $a^2 \equiv 1 \pmod 4$ otherwise. In any case,
the same criterion implies that the quadratic twist by $u \pi$ has
good reduction at $2$.

Over $K=\Q(\sqrt{-7})$ we must also consider the equation
\begin{equation}
  a^2 = \pm T +UP
\end{equation}
(up to Galois conjugation and $2$-isogeny) where $T=\alpha^6$ with
$\alpha=(1+\sqrt{-7})/2$ and $U=\overline{T}$ so that $TU=64$; here
$P$ again denotes a prime power.  The minus sign is impossible
modulo~$\overline{\alpha}^2$, and with the plus sign we can factor as
$(a-\alpha^3)(a+\alpha^3)=UP$.  Arguing as in earlier cases one finds
that this equation has no solutions.

Lastly we consider curves which are supersingular at $\q \mid 2$,
which by Theorem \ref{thm:classification2torsion} and
Corollary~\ref{Cor:ab-q-valuations} arise from solutions to the
following equations:
\begin{eqnarray}
  \tilde{a}^2=P+1 \label{eqn:sscase1}\\
  \tilde{a}^2=P+\epsilon \label{eqn:sscase2}
  \end{eqnarray}
with $a=2\tilde{a}$.

(\ref{eqn:sscase1}) factors as $(\tilde{a}-1)(\tilde{a}+1)$, and one
of the factors is a unit. This gives solutions $P=3$ and $P=-1\pm 2i$
over $\Q(i)$ but the associated curves with $(a,b)=(4,1)$ and
$(2\pm2i,1)$ have bad reduction at~$1+i$ as do all their quadratic
twists.

In (\ref{eqn:sscase2}) the base curve has
$(a,b)=(2\tilde{a},\varepsilon)$ with
$(c_4,c_6)=(2^4(4P+\varepsilon),2^6\tilde{a}(-8P+\varepsilon))$.  We
scale by~$\tau=1+i$ (respectively $\sqrt{-2}$) to
get~$(c_4,c_6)=(\tau^4(4P+i),\tau^6\tilde{a}(8P-i))$ over $\Q(i)$
or~$(c_4,c_6)=(\tau^4(4P-1),\tau^6\tilde{a}(-8P-1))$ over
$\Q(\sqrt{-2})$ respectively.  We must test whether these, or their
twists by $s\in\{1,\varepsilon,\pi,\varepsilon\pi\}$ have good
reduction at~$\tau$.  Note that $\vq(c_4)=4$ and $\vq(c_6) \ge 7$, and
that we are in the second case of Proposition~\ref{prop:Kraus}, with
$a_1=\tau$ (since we are in the supersingular case).

Over $\Q(i)$ the first congruence in Proposition~\ref{prop:Kraus}
reduces to $1+is^2\equiv0\pmod2$ which is impossible.

Over $\Q(\sqrt{-2})$, in the notation of Proposition~\ref{prop:Kraus}
the first condition on~$d$ is always satisfied (since $s$ is odd),
while the second is that either $(1-s^2)/2$ or $(1-s^2)/2+2\tau$ is a
square modulo~$4$, depending on whether $\vq(\tilde{a})\ge2$ or
$\vq(\tilde{a})=1$.  At least one of these is satisfied provided that
$s\equiv\pm1\pmod{\tau^3}$, and in either case $d/16\equiv0\pmod4$.
But now the final condition implies $s^2\equiv-1\pmod4$,
contradiction.
\end{proof}

The above classification implies the following crucial fact, used in
the main theorem of the paper, and which was an important motivation
for this section.

\begin{coro}
  Every isogeny class of elliptic curves defined over $K$ with 
  prime conductor and a $K$-rational $2$-torsion point contains
  a curve whose discriminant has odd valuation.
\label{coro:2torsion-odddiscriminant}
\end{coro}

\begin{proof}
  If $E$ is a curve over $K$ of prime conductor $\id{p}$ and a
  $K$-rational $2$-torsion point, we are in the multiplicative
  case. By Theorem \ref{thm:multiplicativecase} $E$ is either a
  sporadic curve of conductor norm $17$, $241$ over $\Q(\sqrt{-1})$ or
  $257$ over $\Q(\sqrt{-3})$ (all of these have a curve with prime
  discriminant  in their isogeny class) or is isogenous to $E_{a,b}$ with
  $b=16 \varepsilon$ and $a^2=u \pi^r + 64 \varepsilon$ with $r$
  odd. Such curves have discriminant $2^8 u\varepsilon^2 \pi^r$, so odd valuation.
\end{proof}

\begin{remark}
  The computations done in this section could be generalized to other
  number fields of class number one, as the number of units modulo
  squares is always finite. The case of real quadratic fields is of
  particular interest, requiring almost no modification except to
  allow for $\Om_K^*/(\Om_K^*)^2$ having order~$4$. In this case,
  Proposition~\ref{prop:sporadic-family} gives a possibly infinite family of
  elliptic curves with bad reduction only at the primes dividing
  $u^k+16$, where $u$ is a fundamental unit. For example, over
  $\Q(\sqrt{5})$, we get curves of prime conductor with norms
  $1009, 35569, 1659169, \ldots$,  but to our knowledge it is not known
  whether we can get infinitely many curves of prime conductor in this
  way.
\end{remark}

We end this section with an interesting phenomenon concerning curves of prime
conductor and rational $2$-torsion.

\begin{thm}
  Let $K$ be an imaginary quadratic field with class number $1$, and
  $E/K$ be an elliptic curve of prime power conductor with a
  $K$-rational $2$-torsion point. Then $E$ has rank $0$.
\end{thm}

\begin{proof}
  A simple $2$-descent computation shows that this is the case for
  curves in the Setzer-Neumann family (this phenomenon also occurs for
  rational elliptic curves, and the proof is the same).  The remaining
  sporadic cases can be handled by looking at tables \cite{lmfdb} or
  computing the rank of the curve \lmfdbec{27}{a}{1} over the
  different fields $K$ directly, for example using SageMath
  \cite{sage} (using a Pari/GP implementation due to Denis Simon based on the
  article \cite{Simon}).
\end{proof}

\section{On some irreducible finite flat group schemes over $\Spec(\Om_K)$}
\label{section:evenexponents}
Let $E/K$ be a modular elliptic curve of odd prime conductor, whose
discriminant is a square. Then by Theorem~\ref{thm:Mazur} $E[2]$ is a
finite flat group scheme over $\Spec(\Om_K)$ of type $(2,2)$. It is
either reducible or irreducible. In the reducible case, it contains a
factor isomorphic to the multiplicative or the additive group
(\cite[Corollary page 21]{Oort-Tate}), hence the curve has a point of
order $2$ as studied in the previous section. The group
$\Aut_{G_K}(E[2])$ cannot be isomorphic to the whole of $\GL_2(\F_2)$ as
proved in Theorem~\ref{thm:mainthm}, so the only remaining possibility for it is the cyclic group
of order $3$. Note that such a group scheme $E[2]$ does not occur over
$\Spec(\ZZ)$ as there are no cubic extensions of $\Q$ unramified
outside $2$.  The same is true for four of the nine imaginary
quadratic fields under consideration here.

\begin{lemma}
  There are no elliptic curves $E/K$ of odd prime conductor and even
  discriminant valuation whose residual $2$-adic Galois representation
  is cyclic of order $3$ for $K=\Q(\sqrt{-d})$, $d=1, 2, 3$ or $7$.
\label{lemma:cyclic2image}
\end{lemma}
\begin{proof}
  By the aforementioned result of Mazur, the extension $L/K$ obtained
  by adjoining the $2$-torsion points of $E$ is a cyclic cubic
  extension unramified outside $2$. It is easy to verify (for example
  using explicit class field theory as implemented in \cite{PARI2})
  that there is no such extension for these particular fields $K$.
\end{proof}
For the remaining fields $K=\Q(\sqrt{-d})$ for
$d\in\{11,19,43,67,163\}$, there is a unique cyclic cubic extension
$L/K$ unramified outside $2$, namely the ring class field associated
to the order $\Z[\sqrt{-d}]$ of index~$2$, which has class number~$3$ (since $2$ is
inert).  These are the splitting fields of the following
polynomials~$p(x)$ of discriminant $-4d$:
\begin{enumerate}[(i)]
\item $p(x) = x^3-x^2+x+1$ over $\Q(\sqrt{-11})$,
\item $p(x) = x^3-2x-2$ over $\Q(\sqrt{-19})$,
\item $p(x) = x^3-x^2-x+3$ over $\Q(\sqrt{-43})$,
\item $p(x) = x^3-x^2-3x+5$ over $\Q(\sqrt{-67})$,
\item $p(x) = x^3-8x+10$ over $\Q(\sqrt{-163})$.
\end{enumerate}
In particular the splitting field of $E[2]$ must be one of these
fields~$L$. Moreover, for each of these five values of~$d$ there is an
elliptic curve $E$ defined over $\Q$ with prime conductor $d$ and
discriminant $-d$, namely \lmfdbec{11}{a}{3}, \lmfdbec{19}{a}{3},
\lmfdbec{43}{a}{1}, \lmfdbec{67}{a}{1} and \lmfdbec{163}{a}{1}.  In
each case the base-change to $\Q(\sqrt{-d})$ has prime conductor
$(\sqrt{-d})$ and square discriminant $-d=\sqrt{-d}^2$, and the
$2$-division field is the splitting field of the corresponding cubic
$p(x)$.

We can construct further examples over each field as follows.  Rubin
and Silverberg showed in \cite[Theorem 1]{R-S} how to parametrize all
elliptic curves with given level~$2$ structure: given one curve
$E:y^2=x^3+ax+b$, all curves with residual $2$-adic representation
isomorphic to that of $E$ are obtained by specializing the family of
curves
\begin{align}
  \label{eq:residual2-rep}
  y^2&=x^3+A_4(u,v)x+A_6(u,v)\\
  \noalign{\text{with}}
 A_4(u,v)&=3(3av^2+9buv-a^2u^2),\nonumber\\
 A_6(u,v)&=27bv^3-18a^2uv^2-27abu^2v-(2a^3+27b^2)u^3\nonumber
\end{align}
at a pair of elements $(u,v)$ of $K$.  Note that scaling $u,v$ by
$c\in K^*$ gives the quadratic twist by~$c$; hence, up to isomorphism,
we may assume that $u,v\in\Om_K$ with square-free gcd.  The
discriminant of (\ref{eq:residual2-rep}) is $2^43^6\Delta(F)F^2$ where
$F(u,v)=v^3+avu^2+bu^3$ with discriminant $\Delta(F)=-(4a^3+27b^2)$.

\begin{thm}Let $K$ be an imaginary quadratic field of class number $1$
  and $E/K$ be a semistable elliptic curve whose residual $2$-adic
  representation has image cyclic of order $3$. Then $K=\Q(\sqrt{-d})$
  for $d=11, 19, 43, 67, 163$ and the valuation of $\Disc(E)$ at each
  prime of bad reduction is exactly $2$.
\label{thm:absred2}
\end{thm}
\begin{proof}
  The first condition comes from Lemma~\ref{lemma:cyclic2image}.
  Secondly, by the proof of Theorem~\ref{thm:mainthm} the valuation of
  $\Disc(E)$ is a power of~$2$.  It remains to show that the
  discriminant cannot be a fourth power: to see this, we use the
  parametrized families (\ref{eq:residual2-rep}), recalling that $2$
  is inert in each case.

  From (\ref{eq:residual2-rep}) for each field, after a little
  simplification we get $A_4 = -6 P(u,v)$ and $A_6 = 2 Q(u,v)$, with
  discriminant $1728(8P(u,v)^3 - Q(u,v)^2)$ where
\begin{enumerate}
\item $P=2u^2-17uv-v^2$, $Q=586u^3+102u^2v+12uv^2-17v^3$ for $d=11$,
\item $P=2u^2+9uv+3v^2$, $Q=46u^3+54u^2v+36uv^2+27v^3$ for $d=19$,
\item $P=8u^2-35uv+2v^2$, $Q=-2386u^3+420u^2v-48uv^2+35v^3$ for $d=43$,
\item $P=50u^2-53uv+5v^2$, $Q=-4618u^3+1590u^2v-300uv^2+53v^3$ for $d=67$,
\item $P=32u^2+45uv+12v^2$, $Q=838u^3+1080u^2v+576uv^2+135v^3$ for $d=163$.
\end{enumerate}

If $E$ has good reduction at $2$, since $\Disc(E)$ is even, it must be
divisible by $2^{12}$, so $4 \mid Q(u,v)$. Since $Q(u,v) \equiv v^3
\pmod 2$, $v$ must be even, and the condition $4 \mid Q(u,v)$ implies
that $u$ is even as well. After the substitution $(u,v) \to
(u/2,v/2)$, the new $Q(u,v)$ is odd (from the minimality condition)
with invariants $c_4=72P(u,v)$ and $c_6=-216Q(u,v)$. Clearly $v_2(c_4)
\ge 3$, but if it equals 3, we cannot get good reduction at 2 by
Kraus's criterion since the conditions
\begin{eqnarray*}
d= -a_1^6 +3a_1^2c_4 + 2c_6 &\equiv& 0 \pmod{16}\\
4a_1^2 d &\equiv& (a_1^4-c_4)^2 \pmod{256},
\end{eqnarray*}
are not compatible. The first one implies that $a_1$ is even, hence
the left hand side of the second equation is zero, while the right
hand side is not. Then $2 \mid P(u,v)$ and $v_2(c_4) \ge 4$. Kraus's
criterion now implies that there exists $a_1$ such that
$a_1^2 \equiv c_6/8 = -27Q(u,v) \equiv Q(u,v)\pmod 4$.  The
discriminant is now $27(8P^3-Q^2)\equiv 5Q^2 \equiv 5a_1^4\pmod8$,
which cannot be a fourth power since $5$ is not a fourth power
modulo~$8$.
\end{proof}

A natural question is whether there are infinitely many curves of
prime conductor, whose discriminant is a prime square. They are all
obtained by evaluating the previous equations at suitable pairs
$(u,v)$. The model described above has discriminant
$-2^63^6dF(u,v)^2$, where $F(u,v)$ is an explicit cubic form of
discriminant~$-4d$.  The values of $(u,v)$ to get good reduction at
$2$, $3$ and $\sqrt{-d}$ are given by congruence conditions, each one
giving a potentially infinite family, where one expects the cubic
$F(u,v)$ to attain infinitely many prime values.

To end the paper we give an example of an elliptic curve of this type
over each field, in addition to the base-change examples given above.

\begin{enumerate}
  \item
    $K=\Q(\sqrt{-11})=\Q(\alpha)$ where $\alpha^2-\alpha+3=0$:
    $$E:\quad y^2+y=x^3+\alpha x^2-x$$
    (with LMFDB label\footnote{The other curves here do not yet have
    LMFDB labels.} \lmfdbecnf{2.0.11.1}{47.1}{a}{1}) has prime
    conductor $\p=(\pi)$ with $\pi=7-2\alpha$ of norm $47$, and
    discriminant~$\pi^2$.
  \item
    $K=\Q(\sqrt{-19})=\Q(\alpha)$ where $\alpha^2-\alpha+5=0$:
    $$E:\quad y^2+y=x^3+(-\alpha-1)x^2+(2\alpha)x+(-\alpha-1)$$ has
    prime conductor $\p=(\pi)$ with $\pi=18\alpha-7$ of norm $1543$, and
    discriminant~$\pi^2$.
  \item
    $K=\Q(\sqrt{-43})=\Q(\alpha)$ where $\alpha^2-\alpha+11=0$:
    $$E:\quad y^2+y=x^3+(\alpha-1)x^2+(-\alpha-2)x+2$$ has
    prime conductor $\p=(\pi)$ with $\pi=29-2\alpha$ of norm $827$, and
    discriminant~$\pi^2$.
  \item
    $K=\Q(\sqrt{-67})=\Q(\alpha)$ where $\alpha^2-\alpha+17=0$:
    $$E:\quad y^2+y=x^3+(\alpha+1)x^2+2\alpha x+(\alpha-1)$$ has
    prime conductor $\p=(\pi)$ with $\pi=6\alpha-65$ of norm $4447$, and
    discriminant~$\pi^2$.
  \item
    $K=\Q(\sqrt{-163})=\Q(\alpha)$ where $\alpha^2-\alpha+41=0$:
    $$E:\quad y^2+y=x^3+(\alpha+1)x^2+(\alpha-18)x+(-3\alpha-4)$$ has
    prime conductor $\p=(\pi)$ with $\pi=47+6\alpha$ of norm $3967$, and
    discriminant~$\pi^2$.
\end{enumerate}


\begin{thebibliography}{10}

\bibitem{Ash-Stevens}
Avner Ash and Glenn Stevens.
\newblock Cohomology of arithmetic groups and congruences between systems of
  {H}ecke eigenvalues.
\newblock {\em J. Reine Angew. Math.}, 365:192--220, 1986.

\bibitem{Berger}
Tobias Berger and Gergely Harcos.
\newblock {$l$}-adic representations associated to modular forms over imaginary
  quadratic fields.
\newblock {\em Int. Math. Res. Not. IMRN}, 23:Art. ID rnm113, 16, 2007.

\bibitem{Calegari-Geraghty}
Frank Calegari and David Geraghty.
\newblock Modularity lifting beyond the {T}aylor-{W}iles method.
\newblock pages 1--137, 07 2017.

\bibitem{Calegari}
Frank Calegari and Akshay Venkatesh.
\newblock A torsion {J}acquet-{L}anglands correspondence.
\newblock {\em arXiv:1212.3847}, 2012.

\bibitem{CremonaTessellations}
J.~E. Cremona.
\newblock Hyperbolic tessellations, modular symbols, and elliptic curves over
  complex quadratic fields.
\newblock {\em Compositio Math.}, 51(3):275--324, 1984.

\bibitem{CremonaTwist}
J.~E. Cremona.
\newblock Abelian varieties with extra twist, cusp forms, and elliptic curves
  over imaginary quadratic fields.
\newblock {\em J. London Math. Soc. (2)}, 45(3):404--416, 1992.

\bibitem{CremonaAlgo}
J.~E. Cremona.
\newblock {\em Algorithms for modular elliptic curves}.
\newblock Cambridge University Press, Cambridge, second edition, 1997.

\bibitem{CremonaLingham}
J.~E. Cremona and M.~P. Lingham.
\newblock Finding all elliptic curves with good reduction outside a given set
  of primes.
\newblock {\em Experiment. Math.}, 16(3):303--312, 2007.

\bibitem{Cremona-Whitley}
J.~E. Cremona and E.~Whitley.
\newblock Periods of cusp forms and elliptic curves over imaginary quadratic
  fields.
\newblock {\em Math. Comp.}, 62(205):407--429, 1994.

\bibitem{DGP}
Luis Dieulefait, Lucio Guerberoff, and Ariel Pacetti.
\newblock Proving modularity for a given elliptic curve over an imaginary
  quadratic field.
\newblock {\em Math. Comp.}, 79(270):1145--1170, 2010.

\bibitem{Fontaine}
Jean-Marc Fontaine.
\newblock Il n'y a pas de vari\'et\'e ab\'elienne sur {${\bf Z}$}.
\newblock {\em Invent. Math.}, 81(3):515--538, 1985.

\bibitem{TaylorI}
Michael Harris, David Soudry, and Richard Taylor.
\newblock {$l$}-adic representations associated to modular forms over imaginary
  quadratic fields. {I}. {L}ifting to {${\rm GSp}_4({\bf Q})$}.
\newblock {\em Invent. Math.}, 112(2):377--411, 1993.

\bibitem{Ireland-Rosen}
K.~Ireland and M.~Rosen.
\newblock {\em A classical introduction to modern number theory}.
\newblock Number~84 in Graduate Texts in Mathematics. Springer-Verlag, 1982.

\bibitem{Kamienny}
S.~Kamienny.
\newblock Torsion points on elliptic curves and {$q$}-coefficients of modular
  forms.
\newblock {\em Invent. Math.}, 109(2):221--229, 1992.

\bibitem{Kenku}
M.~A. Kenku and F.~Momose.
\newblock Torsion points on elliptic curves defined over quadratic fields.
\newblock {\em Nagoya Math. J.}, 109:125--149, 1988.

\bibitem{Angelos}
A.~Koutsianas.
\newblock Computing all elliptic curves over an arbitrary number field with
  prescribed primes of bad reduction.
\newblock {\em Experimental Mathematics}, pages 1--15 (electronic), 2017.

\bibitem{Kraus-c4c6}
Alain Kraus.
\newblock Quelques remarques \`a propos des invariants $c_4$, $c_6$ et
  {$\Delta$} d'une courbe elliptique.
\newblock {\em Acta Arith.}, 54:75--80, 1989.

\bibitem{Kraus}
Alain Kraus.
\newblock Courbes elliptiques semi-stables et corps quadratiques.
\newblock {\em J. Number Theory}, 60(2):245--253, 1996.

\bibitem{Kubert}
Daniel~Sion Kubert.
\newblock Universal bounds on the torsion of elliptic curves.
\newblock {\em Proc. London Math. Soc. (3)}, 33(2):193--237, 1976.

\bibitem{lmfdb}
The {LMFDB Collaboration}.
\newblock The {L}-functions and {M}odular {F}orms {D}atabase.
\newblock \url{http://www.lmfdb.org}, 2017.
\newblock [Online; accessed 11 December 2017].

\bibitem{Mazur2}
Barry Mazur.
\newblock Rational points of abelian varieties with values in towers of number
  fields.
\newblock {\em Invent. Math.}, 18:183--266, 1972.

\bibitem{Mestre}
J.-F. Mestre and J.~Oesterl{\'e}.
\newblock Courbes de {W}eil semi-stables de discriminant une puissance
  {$m$}-i\`eme.
\newblock {\em J. Reine Angew. Math.}, 400:173--184, 1989.

\bibitem{Miyawaki}
Isao Miyawaki.
\newblock Elliptic curves of prime power conductor with {${\bf Q}$}-rational points of finite order.
\newblock {\em Osaka J. Math.}, 10:309--323, 1973.

\bibitem{PARI2}
The PARI~Group, Bordeaux.
\newblock {\em PARI/GP version {\tt 2.7.0}}, 2014.
\newblock available from \url{http://pari.math.u-bordeaux.fr/}.

\bibitem{Ribetlowering}
Kenneth~A. Ribet.
\newblock Lowering the levels of modular representations without multiplicity
  one.
\newblock {\em Internat. Math. Res. Notices}, (2):15--19, 1991.

\bibitem{R-S}
K.~Rubin and A.~Silverberg.
\newblock Mod {$2$} representations of elliptic curves.
\newblock {\em Proc. Amer. Math. Soc.}, 129(1):53--57, 2001.

\bibitem{Scholze}
Peter Scholze.
\newblock On torsion in the cohomology of locally symmetric varieties.
\newblock {\em Ann. of Math. (2)}, 182(3):945--1066, 2015.

\bibitem{Shumbusho}
  Ren\'e-Michel Shumbusho.
  \newblock Elliptic Curves with Prime Conductor and a Conjecture of Cremona
  \newblock {\em Ph.D thesis, University of Georgia}, 2004.

\bibitem{Haluk-expmath}
Mehmet~Haluk {\c Seng\"un}.
\newblock On the integral cohomology of {B}ianchi groups.
\newblock {\em Exp. Math.}, 20(4):487--505, 2011.

\bibitem{Sengun}
Mehmet~Haluk {\c Seng\"un}.
\newblock Arithmetic aspects of {B}ianchi groups.
\newblock In {\em Computations with modular forms}, volume~6 of {\em Contrib.
  Math. Comput. Sci.}, pages 279--315. Springer, Cham, 2014.

\bibitem{Serre2}
Jean-Pierre Serre.
\newblock Propri\'et\'es galoisiennes des points d'ordre fini des courbes
  elliptiques.
\newblock {\em Invent. Math.}, 15(4):259--331, 1972.

\bibitem{Setzer}
Bennett Setzer.
\newblock Elliptic curves of prime conductor.
\newblock {\em J. London Math. Soc. (2)}, 10:367--378, 1975.

\bibitem{Silverman2}
Joseph~H. Silverman.
\newblock {\em Advanced topics in the arithmetic of elliptic curves}, volume
  151 of {\em Graduate Texts in Mathematics}.
\newblock Springer-Verlag, New York, 1994.

\bibitem{Silverman}
Joseph~H. Silverman.
\newblock {\em The arithmetic of elliptic curves}, volume 106 of {\em Graduate
  Texts in Mathematics}.
\newblock Springer, Dordrecht, second edition, 2009.

\bibitem{sage}
W.\thinspace{}A. Stein et~al.
\newblock {\em {S}age {M}athematics {S}oftware ({V}ersion 8.0)}.
\newblock The Sage Development Team, 2017.
\newblock {\tt http://www.sagemath.org}.

\bibitem{Oort-Tate}
John Tate and Frans Oort.
\newblock Group schemes of prime order.
\newblock {\em Ann. Sci. \'Ecole Norm. Sup. (4)}, 3:1--21, 1970.

\bibitem{Simon}
  Denis Simon.
  \newblock Computing the Rank of Elliptic Curves over Number Fields
  \newblock {\em LMS Journal of Computation and Mathematics}, 5:7--17, 2002.

\bibitem{TaylorII}
Richard Taylor.
\newblock {$l$}-adic representations associated to modular forms over imaginary
  quadratic fields. {II}.
\newblock {\em Invent. Math.}, 116(1-3):619--643, 1994.

\end{thebibliography}
\bibliographystyle{plain}

\end{document}